\documentclass{amsart}
\usepackage{hyperref}
\hypersetup{bookmarksnumbered} % Adds numbers to the section titles in the PDF table of contents
\usepackage{amssymb,amsmath,amsthm}
\usepackage{stackengine} % for stacked text
\usepackage{graphicx}
\usepackage{ulem}
\usepackage{mathtools}
\usepackage[subrefformat=parens,labelformat=parens]{subfig}

\input xy
\xyoption{all}

\usepackage{color}

\makeatletter

\let\c@table\c@figure
\makeatother

\newtheorem{theorem}{Theorem}[section]
\newtheorem{proposition}[theorem]{Proposition}
\newtheorem{lemma}[theorem]{Lemma}
\newtheorem{corollary}[theorem]{Corollary}

\theoremstyle{definition}

\newtheorem{example}[theorem]{Example}

\newtheorem{notation}[theorem]{Notation}

\theoremstyle{remark}
\newtheorem{remark}[theorem]{Remark}

\newcommand{\newword}[1]{\textbf{#1}}
\newcommand{\Orb}{\mathrm{Orb}}
\newcommand{\Tbar}{\mkern 1.5mu\overline{\mkern-1.5mu T}}

\newcommand{\Z}{\mathbb{Z}}
\newcommand{\N}{\mathbb{N}}
\newcommand{\R}{\mathbb{R}}

\newcommand{\Q}{\mathbb{Q}}
\newcommand{\C}{\mathfrak{C}}
\newcommand{\T}{\mathcal{T}}

\DeclareMathOperator{\Homeo}{Homeo}
\DeclareMathOperator{\Diff}{Diff}
\DeclareMathOperator{\PDiff}{PDiff}
\newcommand{\sing}{\mathrm{sing}}
\DeclareMathOperator{\Stab}{Stab}

\newcommand{\Aut}{\mathrm{Aut}}
\newcommand{\TA}{T\!\hspace{0.02em}\mathcal{A}}
\newcommand{\VA}{V\!\!\hspace{0.07em}\mathcal{A}}
\newcommand{\A}{\mathcal{A}}
\newcommand{\F}{\mathrm{F}}

\newcommand{\Finfty}{\mathrm{F}_{\!\hspace{0.04em}\infty}}
\newcommand{\Fn}{\mathrm{F}_{\!\hspace{0.08333em}n}}

\newcommand{\SingFix}{\mathrm{SingFix}}
\newcommand{\SingStab}{\mathrm{SingStab}}
\newcommand{\Fix}{\mathrm{Fix}}

\newcommand{\nc}[1]{\langle\!\langle#1\rangle\!\rangle}

\newcommand{\Bgerm}{\mathrm{Bgerm}}

\title{Finite Germ Extensions}

\author{James Belk}
\address{School of Mathematics \& Statistics, University of Glasgow, Glasgow, G12~8QQ, Scotland.}
\email{\href{mailto:jim.belk@glasgow.ac.uk}{jim.belk@glasgow.ac.uk}}

\thanks{The first author has been partially supported by EPSRC grant EP/R032866/1 during the creation of this paper, as well as the National Science Foundation under Grant No.\ DMS-1854367.}

\author{James Hyde}
\address{Department of Mathematical Sciences, University of Copenhagen, Copenhagen, Denmark.}
\email{\href{mailto:jth@math.ku.dk}{jth@math.ku.dk}}

\author{Francesco Matucci}
\address{Dipartimento di Matematica e Applicazioni, Universit\`{a} degli Studi di Milano--Bicocca, Milan 20125, Italy.}
\email{\href{mailto:francesco.matucci@unimib.it}{francesco.matucci@unimib.it}}
\thanks{The third author is a member of the Gruppo Nazionale per le Strutture Algebriche, Geometriche e le loro Applicazioni (GNSAGA) of the Istituto Nazionale di Alta Matematica (INdAM) and gratefully acknowledges the support of the 
of the Universit\`a degli Studi di Milano--Bicocca
(FA project 2021-ATE-0033 ``Strutture Algebriche'').
}

\date{}  % Removes displayed date

\begin{document}

\maketitle

\begin{abstract}We prove finiteness properties for groups of homeomorphisms that have finitely many ``singular points'', and we describe the normal structure of such groups.  As an application, we prove that every countable abelian group can be embedded into a finitely presented simple group, verifying the Boone--Higman conjecture for countable abelian groups.  Indeed, we describe a specific 2-generated, $\Finfty$ simple group $\VA$ of homeomorphisms of the Cantor set that contains every countable abelian group. As a second application, we prove that if $G$ is a bounded automata group then the associated R\"over--Nekrashevych groups $V_{d,r}G$ have type~$\Finfty$, verifying a conjecture of Nekrashevych for a large class of contracting self-similar groups.  Among others, this result applies to R\"{o}ver--Nekrashevych groups associated to Gupta--Sidki groups and the basilica group.
\end{abstract}

\section{Introduction}

Many naturally defined groups of homeomorphisms have the property that membership in the group is determined locally.  For example, we can determine whether a homeomorphism $f$ of $[0,1]$ lies in the group $\mathrm{Diff}([0,1])$ of all diffeomorphisms of $[0,1]$ by checking whether $f$ is smooth on a small neighborhood of each point.  In general, we say that a homeomorphism group $G\leq\Homeo(X)$ is \newword{full} if every homeomorphism of $X$ that locally agrees with $G$ is an element of~$G$, where a homeomorphism \newword{locally agrees} with $G$ if every point has a neighborhood on which it agrees with some element of~$G$.

Examples of full groups include the group of all $C^r$\mbox{-}diffeomorphisms of any \mbox{$C^r$-manifold} ($1\leq r\leq\infty$), the group of all piecewise-linear homeomorphisms of any PL manifold, and so forth.  Many discrete groups of homeomorphisms are also full, such as Thompson's groups $F$, $T$, and $V$ acting on the interval, the circle, and the Cantor set, respectively.  The terminology ``full'' comes from the theory of \'etale groupoids, where a group is full if and only if it is the topological full group of the corresponding \'etale groupoid of germs~\cite{MatuiSurvey}. Topological full groups of \'etale groupoids are an active area of research~\cite{Matui1, Matui2,Matui3, Nek1, Nek2}.

In this paper we consider what happens when we enlarge a full group by allowing homeomorphisms that fail to locally agree with the group at finitely many points.  For this purpose, we fix a \newword{base group} $B$ of homeomorphisms of a Hausdorff space~$X$.  If $h\in \Homeo(X)$, the \newword{singular points} of $h$ are the set
\[\def\stackalignment{l}
\sing(h) = \biggl\{p\in X \;\biggl|\;\, \text{\stackanchor{$h$ does not agree with any element}{of $B$ on any neighborhood of $p$}}\biggr\}.
\]
A group $G\leq \Homeo(X)$ will be called a \newword{finite germ extension}\footnote{This terminology uses the word ``extension'' in the sense of ``field extension'', i.e.~something into which a given object embeds, as opposed to the usual notion of ``extension'' in group theory.} of $B$ if it satisfies the following conditions:
\begin{enumerate}
    \item Every element of $G$ has only finitely many singular points.
    \smallskip
    \item $G$ contains $B$, and indeed $B$ is precisely the subgroup of $G$ consisting of all elements that have no singular points.\smallskip
    \item For every $g\in G$ and $p\in\sing(g)$, there exists an $h\in G$ with $\sing(h)=\{p\}$ such that $h$ agrees with $g$ on some neighborhood of~$p$.
\end{enumerate}
We do not actually require either of the groups $B$ and $G$ to be full, though condition (2) asserts that $B$ is ``relatively full'' inside of~$G$.  In most cases of interest both $B$ and $G$ will be full groups, though for technical reasons it sometimes helps to consider cases where neither $B$ nor $G$ is full (e.g.\ if we wish to restrict to compactly supported homeomorphisms, as in Section~\ref{subsec:PDiff}).
If $G$ is a finite germ extension of~$B$, the set
\[
\sing(G) = \bigcup_{g\in G} \sing(g)
\]
will be called the \newword{singular set} of~$G$.
Note that $\sing(G)$ is always countable if $G$ is a countable group.

Examples of finite germ extensions are abundant in the literature, and the theory developed here allows us to prove finiteness properties and/or analyze the normal structure of many different groups. Some examples of finite germ extensions include:
\begin{enumerate}
    \item R\"over's group \cite{Rov}
    is a finite germ extension of Thompson's group~$V$.  More generally, if $G\leq \Aut(\T_d)$ is a bounded automata group, then the  R\"over--Nekrashevych group $V_d G$ is a finite germ extension of~$V_d$. 
 (See Section~\ref{sec:RoverNekrashevych} for a definition of these groups.)\smallskip
    \item The group $\PDiff_0^r(X)$ of compactly supported  piecewise-$C^r$-diffeomorphisms of a connected 1-manifold $X$ is a finite germ extension of the diffeomorphism group~$\Diff_0^r(X)$.\smallskip
    \item Two new groups $\TA$ and $\VA$ that we define are finite germ extensions of Thompson's groups $T$ and $V$, respectively.  Every countable, torsion-free abelian group embeds into $\TA$, and every countable abelian group embeds into~$\VA$.
\end{enumerate}
We prove that the groups (1) above have type $\Finfty$ (see Theorem~\ref{thm:NekrashevychGroups}), which was already known for R\"{o}ver's group.  We also compute the abelianization of the groups~(2) above for $r=1$ and $3\leq r \leq \infty$ (Theorem~\ref{thm:PiecewiseCr}).  Finally, we prove that the groups~(3) above are simple and have type~$\Finfty$, which verifies the Boone--Higman conjecture for countable abelian groups (Theorem~\ref{thm:BooneHigmanAbelian}). Incidentally, we observe in Section~\ref{subsec:Examples} that the golden ratio Thompson groups $F_\varphi$, $T_\varphi$, and $V_\varphi$ \cite{Cleary2,BNR1,BNR2} are also finite germ extensions, though we do not prove any new results about them.

The rest of this introduction gives background and careful statements of our main results.  In Section~\ref{subsec:IntroFinitenessProperties} we state our main theorem on finiteness properties and discuss some applications.  In Section~\ref{subsec:IntroNormalStructure} we state our main theorem on normal structure and discuss the application to piecewise diffeomorphisms.  Finally, in Section~\ref{subsec:IntroBooneHigman} we define the groups $\TA$ and $\VA$ and discuss their relevance to the Boone--Higman conjecture.

\subsection{Finiteness properties}\label{subsec:IntroFinitenessProperties}
Our first main result concerns finiteness properties of finite germ extensions.  Recall that a group $G$ has \newword{type $\boldsymbol{\F_n}$} if there exists a $K(G,1)$ complex whose $n$-skeleton is finite.  For example, $G$ has type $\F_1$ if and only if it is finitely generated, and $G$ 
has type $\F_2$ if and only if it is finitely presented. A group has \newword{type~$\boldsymbol{\Finfty}$} if it has type~$\F_n$ for all~$n$.

Recall also that if $G\leq \Homeo(X)$ and $p\in X$, the \newword{group of germs} of $G$ at $p$ is the quotient
\[
(G)_p = \Stab_G(p)/\mathrm{RStab}_G(p)
\]
where $\mathrm{RStab}_G(p)$ is the subgroup of the stabilizer $\Stab_G(p)$ consisting of elements that are the identity in a neighborhood of~$p$.  We prove the following in Section~\ref{sec:FinitenessProperties}.

\begin{theorem}\label{thm:MainFinitenessTheorem}
Let $G$ be a finite germ extension of some $B\leq \Homeo(X)$, and suppose that:
\begin{enumerate}
    \item $B$ has type\/ $\Finfty$, as does\/ $\Stab_B(M)$ for every finite set $M\subseteq\sing(G)$,\smallskip
    \item The action of $B$ on\/ $\sing(G)$ is oligomorphic, and\smallskip
    \item Either\/ $|(G)_p:(B)_p|<\infty$ for every $p\in\sing(G)$, or\/ $(B)_p\trianglelefteq(G)_p$ and\/ $(G)_p/(B)_p$ has type\/~$\Finfty\rule{0pt}{10pt}$ for every $p\in\sing(G)$.
\end{enumerate}
Then $G$ has type\/ $\Finfty$. \end{theorem}

Here the action of $B$ on $\sing(G)$ is \newword{oligomorphic} if the induced action of $B$ on $\sing(G)^k$ has finitely many orbits for all $k\geq 1$ (see~\cite{Cam}).  Though we have stated this theorem for type~$\Finfty$, we state and prove a more general version for type~$\Fn$ in Section~\ref{sec:FinitenessProperties}.  In the case where $B$ is one of the Higman--Thompson groups~$V_{d,r}$, we verify in Appendix~\ref{sec:AppendixA} that $\Stab_B(M)$ has type $\Finfty$ for every finite set $M$ of rational points, a result which may have some independent interest.

We prove Theorem~\ref{thm:MainFinitenessTheorem} by constructing a cubical complex $K$ on which $G$ acts, which we refer to as the \newword{germ complex}.  We then apply the usual combination of Brown's criterion \cite{Bro1} and Bestvina--Brady Morse theory \cite{BesBra} to obtain finiteness properties.  The germ complex $K$ is simply an infinite product of trees, with one tree for each point of $\sing(G)$, and in particular is a $\mathrm{CAT}(0)$ cubical complex.  This complex is defined using only the structure of the singular points, as opposed to the combinatorics of the groups $B$ and $G$ themselves, and as such and is quite different from complexes which have previously been constructed for such groups.

As an application of Theorem~\ref{thm:MainFinitenessTheorem}, we verify a conjecture of Nekrashevych~\cite{Nek2} for a large class of R\"over--Nekrashevych groups.  Recall that a \newword{self-similar group} is a group $G\leq \Aut(\T_d)$ of automorphisms of an infinite, rooted $d$-ary tree whose image under the natural isomorphism $\Aut(T_d)\to \Aut(\T_d)\wr S_d$ is contained in $G\wr S_d$.  The prototypical examples of self-similar groups are the Grigorchuk group \cite{GrigFirst} and the Gupta--Sidki groups \cite{GuptaSidki}, which were introduced in the 1980's as examples of infinite, finitely generated torsion groups.  Grigorchuk's group was also the first known example of a group with intermediate volume growth~\cite{GrigIntermediate}.
In addition to their significance in group theory, self-similar groups have important connections to complex dynamics~\cite{BNTwistedRabbit,NekBook, NekItMon}, fractal geometry \cite{BGN,Sun}, and \mbox{$C^*$-algebras}~\cite{ExPa,NekCstar}. 

Motivated by connections to \mbox{$C^*$-algebras}, Nekrashevych~\cite{Nek} defined for every self-similar group $G\leq \Aut(\T_d)$ an associated \newword{R\"over--Nekrashevych group} $V_dG$, obtained by combining $G$ with a generalized Thompson group~$V_d$.  Such groups had previously been considered by Scott~\cite{Scott}, and in the case where $G$ is Grigorchuk's group, the resulting group $VG$ is \newword{R\"over's group}~\cite{Rov}. Unlike self-similar groups, which are usually not finitely presented, R\"{o}ver--Nekra\-shevych groups often have good finiteness properties.  R\"over proved that $VG$ is finitely presented, and in 2016 the first and third authors proved that~$VG$ has type~$\Finfty$~\cite{BeMa}.  Skipper and Zaremsky later generalized this result to two infinite classes of R\"over--Nekrashevych groups~\cite{SkZa}, and along with Witzel used this to construct the first family of simple groups with arbitrary finiteness length~\cite{SWZ}.  See~\cite{Li} for a generalization of this result. 

Meanwhile, Nekrashevych has proven that $V_dG$ is finitely presented whenever $G$ is a contracting self-similar group~\cite{Nek2}, and has conjectured that $V_dG$ has type $\Finfty$ in this case\footnote{An innovative approach to Nekrashevych's conjecture announced in 2014~\cite{BaGe,BaGe2} has not yielded a complete proof.}.
We use Theorem~\ref{thm:MainFinitenessTheorem} to verify Nekrashevych's conjecture for a large class of contracting self-similar groups.  In particular, we prove the following in Section~\ref{sec:RoverNekrashevych}

\begin{theorem}\label{thm:NekrashevychGroups}If $G\leq\Aut(\T_d)$ is a bounded automata group, then the corresponding R\"over--Nekrashevych group $V_dG$ has type\/~$\Finfty$.
\end{theorem}

Bounded automata groups are a large and interesting class of contracting self-similar groups~\cite{Bond}.  Indeed, they are precisely the contracting self-similar groups whose associated limit spaces are finitely ramified fractals~\cite{BoNe}.  The class of bounded automata groups includes most well-known examples of self-similar groups, including the Grigorchuk group, the Gupta--Sidki groups~\cite{GuptaSidki}, the basilica group~\cite{GrigZuk}, BSV groups~\cite{BruSidVie}, finite state spinal groups~\cite{BarGriSun}, and the iterated monodromy groups of postcritically finite polynomials~\cite{NekBook}.  For the latter two of these examples, Theorem~\ref{thm:NekrashevychGroups} gives the first proof that the associated R\"over--Nekrashevych groups have type~$\Finfty$.

\subsection{Normal structure}\label{subsec:IntroNormalStructure}
Our second main family of results concerns the normal structure of finite germ extensions.  Epstein and Mather have shown that many full groups are either simple or have simple commutator subgroup \cite{Eps, Eps2, Mather}, and the same holds for many Thompson-like groups. In Section~\ref{ssec:simplicity} we prove the following simplicity result.

\begin{theorem}\label{thm:Simplicity}
Let $G$ be a finite germ extension of some $B\leq\Homeo(X)$. Suppose that $B$ is simple, locally moving, and has no global fixed points, and that $B$ and $G$ have the same orbits in~$X$. Then the commutator subgroup $G'$ is simple, and every proper quotient of $G$ is abelian. \end{theorem}

Here $B$ is \newword{locally moving} if for each nonempty open set $U\subseteq X$, there exists a nontrivial element of $B$ which is supported on~$U$. 

We also compute the abelianization $G/G'$ of $G$. Given a point $p\in \sing(G)$, let $A_p(G)$ be the abelian group that fits into the exact sequence
\[
(B)_p\bigr/(B)_p' \longrightarrow (G)_p\bigr/(G)_p' \longrightarrow A_p(G) \longrightarrow 0.
\]
We prove the following in Section~\ref{ssec:abelianization}.

\begin{theorem}\label{thm:Abelianization}
Let $G$ be a finite germ extension of some $B\leq \Homeo(X)$, and suppose that $B$ has no global fixed points, and that $B$ and $G$ have the same orbits in $X$. Then we have an exact sequence
\[
B/B' \longrightarrow G/G'\longrightarrow \bigoplus_{p\in R} A_p(G) \longrightarrow 0
\]
where $R$ is a system of representatives for the $B$-orbits in $X$.
In particular, if $B$ is perfect then $G/G'\cong \bigoplus_{p\in R} A_p(G)$.
\end{theorem}

Here the hypothesis that $B$ and $G$ have the same orbits is equivalent to requiring that $G$ is generated by $B$ together with homeomorphisms that fix their singular points.

As a simple application of these results, we prove a theorem about piecewise-diffeomorphism groups of $1$-manifolds that does not seem to have previously appeared in the literature.  Given a connected 1-manifold $X$ (i.e.\ $X=\R$ or $X=S^1$), we say that an $f\in\Homeo(X)$ is \newword{piecewise-$\boldsymbol{C^r}$-diffeomorphic} if there exists a locally finite cover of $X$ by closed intervals such that $f$ agrees with a $C^r$ diffeomorphism of $X$ on each interval.  Let $\mathrm{PDiff}_0^r(X)$ be the group of all orientation-preserving, piecewise-$C^r$-diffeomorphic homeomorphisms of $X$ with compact support.  Note that this is a finite germ extension of the group $\mathrm{Diff}_0^r(X,r)$ of all compactly supported $C^r$-diffeomorphisms of~$X$. We prove the following result in Section~\ref{subsec:PDiff}.

\newcommand{\textPiecewiseCr}{%
Let $X$ be a connected\/ $1$-manifold, and let $r=1$ or\/ $3\leq r \leq \infty$.  Let $\varphi\colon \mathrm{PDiff}_0^r(X)\to \R$ be the epimorphism
\[
\varphi(f) = \sum_{p\in X} \log\biggl(\frac{f^+(p)}{f^-(p)}\biggr)
\]
where $f^+(p)$ and $f^-(p)$ denote the one-sided derivatives of~$f$ at~$p$.  Then the commutator subgroup of\/ $\PDiff_0^r(X)$ is simple and is equal to the kernel of~$\varphi$.%
}
\begin{theorem}\label{thm:PiecewiseCr}
\textPiecewiseCr
\end{theorem}

Note that this theorem does not apply to $r=2$, for in this case it is not even known whether $\mathrm{Diff}_0^r(X)$ 
 is simple.

\subsection{Application to the Boone--Higman conjecture}\label{subsec:IntroBooneHigman}
William Boone and Graham Higman conjectured in 1973 that every countable group with solvable word problem embeds into a finitely presented simple group~\cite{BoHi}.  This conjecture has motivated much work in the theory of finitely presented simple groups, including Higman's investigation of the Higman--Thompson groups~$V_{n,r}$ \cite{Hig2}, Elizabeth Scott's construction of finitely presented simple groups containing $\mathrm{GL}_n(\Z)$ \cite{Scott,Scott2}, and Claas R\"{o}ver's construction of a finitely presented simple group that contains Grigorchuk's group~\cite{Rov}.  See \cite{BBMZ1} for a survey of results related to this conjecture.

Since every countable abelian group embeds into $\bigoplus_\omega \bigl(\Q\oplus \Q/\Z \bigr)$, it would follow from the Boone--Higman conjecture that every countable abelian group embeds into a finitely presented simple group.  We verify this aspect of their conjecture
in Section~\ref{sec:BooneHigman}.

\begin{theorem}\label{thm:BooneHigmanAbelian}
Every countable abelian group embeds into a finitely presented simple group.
\end{theorem}

Note that the class of countable abelian groups is much richer than the the class of finitely generated abelian groups.  No classification of such groups is known, and indeed a straightforward classification is most likely impossible.  For example, Downey and Montalb\'an have shown that the isomorphism problem for countable, torsion-free abelian groups is $\Sigma_1^1$-complete~\cite{DoMo}.  It was proven by Higman that any countable abelian torsion group embeds into~$V$ \cite[Theorem~6.6]{Hig}, and Scott proved that $\Z\bigl[\tfrac1n\bigr]$ embeds into a finitely presented simple group for any~$n$ \cite{Scott2}, but no embedding of $\Q$ into a finitely presented simple group was previously known.

Note also that the Boone--Higman conjecture remains open for the class of finitely generated metabelian groups, all of which have solvable word problem~\cite{Baumslag2}.  Baumslag~\cite{Baumslag} and Remeslennikov \cite{Rem} have proven that every such group embeds into a finitely presented metabelian group.

We prove Theorem~\ref{thm:BooneHigmanAbelian} by describing an explicit finitely presented simple group $\VA$ of homeomorphisms of the Cantor set $\C_2=\{0,1\}^\omega$ that contains $\bigoplus_{\omega}(\Q\oplus\Q/\Z)$, and hence contains every countable abelian group.  We also describe a finitely presented simple group $\TA$ of homeomorphisms of the circle $S^1=\R/\Z$ that contains $\bigoplus_\omega \Q$, and hence contains every countable, torsion-free abelian group.

To define $\TA$, recall first that Thompson's group $T$ is the group of all piecewise-linear homeomorphisms of $S^1$ whose breakpoints are dyadic and whose linear segments all have the form $t\mapsto 2^nt+b$ for $n\in\Z$ and $b\in\Z\bigl[\tfrac12\bigr]$.  
The group $\TA$ can be defined as the finite germ extension of $T$ consisting of all $f\in\Homeo(S^1)$ that satisfy the following conditions:
\begin{enumerate}
    \item The singular set $\sing(f)$ is a finite set of dyadic points. Thus $f$ is composed of finitely or infinitely many linear segments, which may accumulate at the points of~$\sing(f)$.\smallskip
    \item If $p\in \sing(f)$, then $L_{f(p)}\circ f$ agrees with $f\circ L_p$ in some neighborhood of~$p$, where $L_p$ denotes the linear map in a neighborhood of $p$ that fixes $p$ and has  slope~$2$. 
\end{enumerate}
The group $\VA$ is defined similarly. 
 In particular, recall that Thompson's group $V$ is the group of homeomorphisms $f$ of $\C_2$ with the following property: for every point $p\in \C_2$, there are finite prefixes $\alpha,\beta$ of $p$ and $f(p)$, respectively, so that $f$ agrees with the prefix replacement $\alpha\psi\mapsto \beta\psi$ in a neighborhood of~$p$.  Then $\VA$ can be defined as the finite germ extension of $V$ consisting of all $f\in\Homeo(\C_2)$ that satisfy the following conditions:
\begin{enumerate}
    \item The singular set $\sing(f)$ is finite, and each point of $\sing(f)$ ends in either $\overline{0}$ or $\overline{1}$.\smallskip
    \item If $p$ is a point in $\sing(f)$, then there exists a neighborhood of $p$ on which $f$ is order-preserving and $L_{f(p)}\circ f$ agrees with $f\circ L_p$; here $L_p$ denotes the prefix replacement $\alpha 0\psi\mapsto \alpha \psi$ if $p=\alpha \overline{0}$, or $\alpha 1\psi \mapsto \alpha \psi$ if $p=\alpha\overline{1}$.
\end{enumerate}
Our finiteness result (Theorem~\ref{thm:MainFinitenessTheorem}) proves that these groups are finitely presented---and indeed have type~$\Finfty$---while our simplicity and abelianization results (Theorems~\ref{thm:Simplicity} and \ref{thm:Abelianization}) prove that $\TA$ and $\VA$ are simple.  The details of this argument are given in Section~\ref{sec:BooneHigman} where we also prove that both $\TA$ and $\VA$ are 2-generated and consider finite presentations for these groups.

Both of the groups $\TA$ and $\VA$ contain a certain group $\A\leq \Homeo_+(\R)$ introduced by Brin~\cite{Brin0}, as we will describe in Section~\ref{subsec:MainTAVA}.  The authors proved in~\cite{BHM} that this group $\A$ contains~$\Q$, making it (along with another group~$\Tbar$) the first natural example of a finitely presented group that contains~$\Q$.  As we describe in Section~\ref{subsec:MainTAVA}, the group $\TA$ contains $\bigoplus_{\omega} \A$ and hence $\bigoplus_{\omega} \Q$, while the group $\VA$ contains $\bigoplus_{\omega} (\A \oplus V)$.  Since $\A$ contains $\Q$ and $V$ contains $\Q/\Z$, it follows that $\VA$ contains $\bigoplus_{\omega} (\Q \oplus \Q/\Z)$, and thus contains every countable abelian group.

\subsection*{Acknowledgments}

The authors would like to thank Collin Bleak, Matthew Brin, Sang-hyun Kim, Thomas Koberda, Nicolas Matte-Bon, Rachel Skipper, Michele Triestino, and Matthew Zaremsky for many helpful discussions. The authors would particularly like to thank Michele Triestino for observing that the germ complex is CAT(0)-cubical, and Rachel Skipper for suggesting we extend Theorem~\ref{thm:NekrashevychGroups} to include all bounded automata groups.

\section{Finiteness Properties}\label{sec:FinitenessProperties}

In this section we define the germ complex and use it to prove Theorem~\ref{thm:MainFinitenessTheorem}.  We begin by defining the germ complex in Section~\ref{sec:GermComplex} and proving finiteness properties for a finite germ extension $G$ in terms of certain subgroups $\SingFix_G(M,M')$.  We then analyze the finiteness properties of these subgroups in Section~\ref{sec:FinitenessSingFix}, yielding a proof of Theorem~\ref{thm:MainFinitenessTheorem}.  Finally, in Section~\ref{sec:RoverNekrashevych} we apply these results to R\"over--Nekrashevych groups, yielding a proof of Theorem~\ref{thm:NekrashevychGroups}.

\subsection{The germ complex}\label{sec:GermComplex}

In this section we define the germ complex and prove the following theorem.

\begin{theorem}\label{thm:SingFixFiniteness}
Let $X$ be Hausdorff space, and let $G$ be a finite germ extension of some $B\leq\Homeo(X)$.  Let~$n\geq 1$, and suppose that:
\begin{enumerate}
    \item The induced action of $B$ on\/ $\sing(G)^n$ has finitely many orbits, and \smallskip
    \item For every pair $M\subseteq M'$ of subsets of\/ $\sing(G)$ with\/ $0\leq |M'| \leq n$, the group
    \[
    \SingFix_G(M,M') \coloneqq \{g\in G \mid \sing(g)\subseteq M\text{ and }g|_{M'}=\mathrm{id}\}
    \]
    has type\/ $\Fn$.
\end{enumerate}
Then $G$ has type\/ $\Fn$.
\end{theorem}

Note that $\SingFix_G(\emptyset,\emptyset)=B$, so condition (2) includes the hypothesis that $B$ has type~$\Fn$.

\begin{notation}
For the remainder of this section, let $X$ be a Hausdorff space, let $B\leq \Homeo(X)$, and let $G$ be a finite germ extension of~$B$.  % We assume that $B$ and $G$ are countable, so in particular the set $\mathrm{sing}(G)$ of all singular points of all elements of~$G$ is a countable set.
\end{notation}

\subsubsection*{Germs}
The \newword{germ} of an element $g\in G$ at a point $p\in X$ is the set
\[
(g)_p = \{h\in G \mid \text{$h$ agrees with $g$ in some neighborhood of $p$}\}.
\]
Equivalently, $(g)_p$ is the left coset $g\,\mathrm{RStab}_G(p)$, where $\mathrm{RStab}_G(p)$ is the group of elements of $G$ that are the identity in a neighborhood of~$p$.

The \newword{inverse} of a germ $(g)_p$ is $(g^{-1})_{g(p)}$.  There is also a partially defined \newword{composition} operation for germs, namely
\[
(g)_{h(p)}(h)_p = (gh)_p
\]
for $g,h\in G$ and $p\in X$.  Note that the group of germs $(G)_p$ defined earlier is precisely the group of all germs $(g)_p$ for which $g(p)=p$.

\subsubsection*{$B$-germs} The \newword{$\boldsymbol{B}$-germ} of an element $g\in G$ at a point $p\in X$ is the set
\[
(Bg)_p = \{(bg)_p \mid b\in B\}.
\]
We let $\Bgerm(G,p)$ denote the set of all $B$-germs of $G$ at~$p$.  Note that this includes the \newword{trivial $\boldsymbol{B}$-germ} $(B)_p$, as well as many non-trivial $B$-germs at~$p$. The following lemma asserts that we can arrange for an element of $G$ to have any finite collection of nontrivial $B$-germs.

\begin{lemma}\label{lem:CompatibleBGerms}Let\/ $(Bg_1)_{p_1},\ldots,(Bg_n)_{p_n}$ be nontrivial $B$-germs, where $p_1,\ldots,p_n$ are distinct points in\/ $\sing(G)$.  Then there exists a $g\in G$ with\/
$\sing(g)= \{p_1,\ldots,p_n\}$ such that\/ $(Bg)_{p_i}=(Bg_i)_{p_i}$ for all~$i$.
\end{lemma}
\begin{proof}
We proceed by induction on $n$.  The base case is $n=0$, for which any element of $B$ suffices.

Now suppose that $n \geq 1$. By our inductive hypothesis, there exists an $h\in G$ with $\sing(h)= \{p_1,\ldots,p_{n-1}\}$ such that $(Bh)_{p_i}=(Bg_i)_{p_i}$ for all $1\leq i \leq n-1$. Since $G$ is a finite germ extension of $B$, there exists a $k\in G$ with  $\sing(k)\subseteq \{h(p_n)\}$ such that $(k)_{h(p_n)}=(g_nh^{-1})_{h(p_n)}$.  Let $g=kh$, and note that
\[
\sing(g) \subseteq h^{-1}\sing(k)\cup \sing(h) \subseteq h^{-1}\{h(p_n)\}\cup \{p_1,\ldots,p_{n-1}\} = \{p_1,\ldots,p_n\}.
\]
Since $k$ agrees with elements of~$B$ in neighborhoods of $h(p_1),\ldots,h(p_{n-1})$, we know that  $(Bg)_{p_i}=(Bh)_{p_i}=(Bg_i)_{p_i}$ for all $1\leq i \leq n-1$.  Furthermore, since 
\[
(g)_{p_n} = (k)_{h(p_n)} (h)_{p_n} = (g_nh^{-1})_{h(p_n)}(h)_{p_n} = (g_n)_{p_n},
\]
we have $(Bg)_{p_n} = (Bg_n)_{p_n}$, and therefore $(Bg)_{p_i}=(Bg_i)_{p_i}$ for all $1 \leq i \leq n$.  Since these $B$-germs are all nontrivial, it follows that $\sing(g)=\{p_1,\ldots,p_n\}$.
\end{proof}

\begin{remark}\label{rem:SingularPointsComposition}
In this proof, we used the fact that
\[
\sing(kh) \subseteq h^{-1}\sing(k)\cup \sing(h)
\]
for any homeomorphisms $h$ and $k$.  We will be using this repeatedly in the future.
\end{remark}

\subsubsection*{Portraits} Define the \newword{portrait} for an element $g\in G$ to be the function that assigns to each point $p\in\sing(G)$ the associated $B$-germ $(Bg)_p$.  By Lemma~\ref{lem:CompatibleBGerms}, any choice function $\gamma\colon \sing(G)\to \bigcup_{p\in\sing(G)} \Bgerm(G,p)$ with finitely many nontrivial $B$\mbox{-}germs is the portrait for some element of~$G$.

There is a natural left action of $G$ on $B$-germs defined by
\[
g\cdot (Bh)_p = 
(Bhg^{-1})_{g(p)}
\]
for all $g,h\in G$ and $p\in X$.  Note here that $g$ maps $\Bgerm(G,p)$ to $\Bgerm(G,g(p))$. This extends to a left action of $G$ on portraits $\gamma$ by the rule
\[
(g\cdot \gamma)(p) = g \cdot \gamma(g^{-1}p).
\]

Note that two elements of $G$ have the same portrait if and only if they lie in the same right coset of the base group~$B$.  Indeed, the left action of $G$ on portraits is induced by the usual right action of $G$ on these cosets.

\subsubsection*{Partial portraits and the complex $K$}

For each $p\in\sing(G)$, let
\[
\Bgerm^*(G,p) = \Bgerm(G,p)\cup \{*\},
\]
where $*$ is a new symbol representing ``no information''.  A \newword{partial portrait} is a choice function $\gamma\colon \sing(G)\to \bigcup_{p\in \sing(G)} \Bgerm^*(G,p)$ with finitely many nontrivial germs and finitely many $*$'s.  If $\gamma$ is a partial portrait, we will refer to the points in $\sing(G)$ that map to $*$ as \newword{hidden points}.  The action of $G$ on portraits extends to an action of $G$ on partial portraits in an obvious way.

Our goal is to define a cubical complex $K$ whose vertices are the partial portraits.  For each $p\in \sing(G)$, let $T_p$ be the tree whose vertices are the elements of $\Bgerm^*(G,p)$, with an edge from the basepoint $(B)_p$ to the vertex $*$, as well as an edge from $*$ to each nontrivial $B$-germ $(Bg)_p$.  Note that $T_p$ has the shape of a star centered at~$*$, but the basepoint of $T_p$ is $(B)_p$, not~$*$. Let $K$ be the (restricted) infinite product $\prod_{p\in \sing(G)} T_p$, i.e.\ the collection of all choice functions $\gamma\colon \sing(G)\to \bigcup_{p\in\sing(G)} T_p$ for which all but finitely many of the choices are the basepoint of~$T_p$.  Then $K$ has the structure of a cubical complex whose vertices are precisely the partial portraits.

Note that $K$ is contractible since it is a product of trees.  Indeed, $K$ is a CAT(0) cubical complex.  The action of $G$ on partial portraits extends naturally to a left action of $G$ on~$K$.  However, this action is not proper.  For example, the stabilizer of the basepoint of $K$ is the entire base group~$B$.

\subsubsection*{Cube stabilizers}

Since the action of $G$ on portraits is transitive, every portrait lies in the $G$-orbit of the basepoint of~$K$.  More generally, we say that a partial portrait is \newword{non-singular} if all of its $B$\mbox{-}germs are trivial.  It follows easily from Lemma~\ref{lem:CompatibleBGerms} that every partial portrait lies in the $G$-orbit of a non-singular partial portrait.  

If $\gamma$ is a non-singular partial portrait whose hidden points are some finite set $M\subset \sing(G)$, then \[
\Stab_G(\gamma) = \{g\in G \mid \sing(g)\subseteq M\text{ and }g(M)=M\}.
\]
Note that this stabilizer has the group
\[
\SingFix_G(M,M) = \{g\in G \mid \sing(g)\subseteq M\text{ and }g|_M=\mathrm{id}\}
\]
as a subgroup of finite index.  In particular, if $\SingFix_G(M,M)$ has type $\Fn$ for every finite set $M\subset \sing(G)$, then all of the vertex stabilizers in $K$ have type~$\Fn$ (see \cite[Proposition~7.2.3]{Geog}).

More generally, we say that a cube in $K$ is \newword{non-singular} if all of its vertices are non-singular partial portraits.  Every non-singular cube is determined by a pair $M\subseteq M'$ of finite subsets of~$\sing(G)$, with the vertices of the cube being all non-singular partial portraits whose set $S$ of hidden points satisfies $M\subseteq S\subseteq M'$.

Again, it follows easily from Lemma~\ref{lem:CompatibleBGerms} that every cube in $K$ lies in the orbit of a non-singular cube.  If $C$ is a non-singular cube corresponding to finite sets $M\subseteq M'$ in $\sing(G)$, then
\[
\Stab_G(C) = \{g\in G \mid \sing(g)\subseteq M\text{, }g(M)=M\text{, and }g(M')=M'\}.
\]
Again, this has
\[
\SingFix_G(M,M') = \{g\in G \mid \sing(g)\subseteq M\text{ and }g\in \Fix_G(M')\}
\]
as a subgroup of finite index.  The following proposition summarizes these findings.

\begin{proposition}\label{prop:CellStabilizers}
If $C$ is any cube in $K$, then\/ $\Stab_G(C)$ has a finite-index subgroup conjugate to\/ $\SingFix_G(M,M')$ for some finite sets $M\subseteq M'$ in\/ $\sing(G)$, where $|M'|$ is the maximum number of $*$'s in any vertex of~$C$.\hfill\qedsymbol
\end{proposition}

\subsubsection*{The Morse function}

Recall that a \newword{Morse function} on an affine cell complex $K$ is a map $\mu\colon K\to [0,\infty)$ that maps the vertices of $K$ to natural numbers and restricts to a non-constant affine linear map on each cell of~$K$.  For the cubical complex $K$ we are considering, there exists a unique Morse function $\mu\colon K \to [0,\infty)$ that assigns to each vertex the number of hidden points in the corresponding partial portrait, with $\mu$ extended linearly over each~cube.  

For each $n\in\mathbb{N}$, the corresponding \newword{sublevel complex} $K_{\leq n}$ is the subcomplex of $K$ consisting of all cubes that are contained in $\mu^{-1}\bigl([0,n]\bigr)$.  Since the Morse function $\mu$ is invariant under the action of~$G$, each of the sublevel complexes $K_{\leq n}$ is $G$-invariant.

\begin{proposition}\label{prop:FinitelyManyOrbitsCells}
For $n\in\mathbb{N}$, the following are equivalent:
\begin{enumerate}
    \item The induced action of $B$ on\/ $\sing(G)^n$ has finitely many orbits.\smallskip
    %\red{\item The action of $B$ on\/ $K_{\leq n}$ has finitely many orbits of non-singular cubes.}
    \item The action of $G$ on $K_{\leq n}$ has finitely many orbits of cubes.
\end{enumerate}\end{proposition}
\begin{proof}
Recall first that every cube in $K_{\leq n}$ is in the $G$-orbit of some non-singular cube.  Furthermore, any element of $G$ mapping one non-singular cube to another must lie in $B$.  Thus the $G$-orbits of cubes in $K_{\leq n}$ are in one-to-one correspondence with the $B$-orbits of non-singular cubes in $K_{\leq n}$.

However, the non-singular cubes in $K_{\leq n}$ are in one-to-one correspondence to pairs of sets $M\subseteq M'$ in $\sing(G)$ with $|M'|\leq n$.  The group $B$ has finitely many orbits of such pairs if and only if it has finitely many orbits on $\sing(G)^n$, and thus the two conditions are equivalent.
\end{proof}

\begin{proposition}\label{prop:nConnected}
For $n\geq 1$, the sublevel complex $K_{\leq n}$ is $(n-1)$-connected.
\end{proposition}
\begin{proof}
We use Bestvina--Brady Morse theory (see~\cite{BesBra}).  Recall that the \newword{descending link} of a vertex $v\in K$ is the link of $v$ in $K_{\leq \mu(v)}$.  Let $v$ be a vertex in $K$, and suppose first that $v$ corresponds to a non-singular portrait whose hidden points are a finite set  $M\subset \sing(G)$.  Then the vertices adjacent to $v$ in $K_{\leq\mu(V)}$ are those obtained by replacing one of the $*$'s by a $B$-germ at that point, and a set of such vertices lie in a common cube with $v$ if and only if the corresponding replaced $*$'s are at different points of~$M$.  It follows that the descending link for~$v$ is isomorphic to the join {\Large $*$}$_{p\in M} \Bgerm(G,p)$, where each $\Bgerm(G,p)$ is a discrete set.  By a theorem of Milnor~\cite[Lemma~2.3]{Mil}, the join of any $k$ nonempty simplical complexes is always $(k-2)$-connected, and therefore the descending link of $v$ is $(|M|-2)$-connected, where $|M|=\mu(v)$.  Since every partial portrait lies in the $G$-orbit of a non-singular partial portrait, it follows that the descending link of any vertex $v$ in $K$ is $(\mu(v)-2)$-connected.

The desired result now follows from a theorem of  Bestvina and Brady \cite[Corollary~2.6]{BesBra}. According to their theorem, the statement about descending links proven above implies that the homomorphism $\widetilde{H}_i(K_{\leq n})\to \widetilde{H}_i(K)$ on reduced homology induced by inclusion is an isomorphism for $i\leq n-1$.  Moreover, for $n\geq 2$ their theorem also yields that the induced homomorphism $\pi_1(K_{\leq n})\to \pi_1(K)$ is an isomorphism.  Since $K$ is contractible, we know that $\pi_1(K)=1$ and $\widetilde{H}_i(K)=0$ for all~$i$,  and therefore $K_{\leq n}$ is $(n-1)$-connected.
\end{proof}

\begin{proof}[Proof of Theorem \ref{thm:SingFixFiniteness}]
Let $n\geq 1$, and suppose that
\begin{enumerate}
    \item The action of $B$ on $\sing(G)^n$ has finitely many orbits, and\smallskip
    \item $\SingFix_G(M,M')$ has type $\Fn$ for all sets $M\subseteq M'\subset \sing(G)$ with $|M'|\leq n$.
\end{enumerate}
We wish to prove that $G$ has type $\Fn$.

It is well known that if $G$ is a group acting cellularly on an $(n-1)$-connected affine cell complex~$K$, then $G$ has type $\Fn$ as long as $K$ has finitely many orbits of cells and the stabilizer of each cell has type~$\Fn$.  (This follows, for example, by applying \cite[Theorem~7.3.1]{Geog} to the barycentric subdivision of~$K$.)

In this case, $G$ acts on the cubical complex $K_{\leq n}$, which is  $(n-1)$-connected by Proposition~\ref{prop:nConnected}.  By condition~(1) and Proposition~\ref{prop:FinitelyManyOrbitsCells}, there are only finitely many orbits of cubes  in~$K_{\leq n}$.  By condition~(2) and Proposition~\ref{prop:CellStabilizers}, the stabilizer of each cube in $K_{\leq n}$ has type~$\Fn$, and therefore $G$ has type~$\Fn$. \end{proof}

%%%%%%%%%%%%%%%%%%%%%%%%%%%%%
%%%%%%%%%%%%%%%%%%%%%%%%%%%%%
% Commented out remark
%%%%%%%%%%%%%%%%%%%%%%%%%%%%%
%%%%%%%%%%%%%%%%%%%%%%%%%%%%%
%\begin{remark}
%By using a certain simplicial subdivision of $K_{\leq n}$ in the proof of Theorem~\ref{thm:SingFixFiniteness}, it is possible to weaken the hypotheses slightly to assuming that:
%\begin{enumerate}
%    \item The action of $B$ on $\sing(G)^n$ has finitely many orbits,\smallskip
%    \item $\SingFix_G(M,M)$ has type $\Fn$ for every set $M\subset \sing(G)$ with $|M|\leq n$, and\smallskip
%    \item $\SingFix_G(M,M')$ has type $\F_{\!\hspace{0.083333em}n-1}$ for all sets $M\subset M'\subset \sing(G)$ with $|M|<|M'|\leq n$.
%\end{enumerate}
%\end{remark}

\subsection{Finiteness properties for \texorpdfstring{$\boldsymbol{\SingFix(M,M')}$}{SingFix(M,M')}}\label{sec:FinitenessSingFix}

Theorem~\ref{thm:SingFixFiniteness} requires knowledge of the finiteness properties for the family of groups $\SingFix_G(M,M')$. In this section we prove two statements about the structure of these groups that allow one to deduce finiteness properties.

Throughout this section, let $X$ be a Hausdorff space, let $B\leq\Homeo(X)$, and let $G$ be a finite germ extension of~$B$. Let $M\subseteq M'$ be finite subsets of~$\sing(G)$, and let $\Fix_{B}(M')$ be the subgroup of $B$ consisting of elements that fix $M'$ pointwise.  Note that $\Fix_B(M')$ has finite index in $\Stab_B(M')$, so $\Fix_B(M')$ has type $\Fn$ if and only if $\Stab_B(M')$ has type $\Fn$ (see \cite[Proposition~7.2.3]{Geog}).

%We will use the following proposition in \blue{Section~??} to show that certain R\"{o}ver--Nekrashevych groups have type~$\Finfty$.

\begin{proposition}\label{prop:FinitnessPropertiesFiniteIndex}
If $(B)_p$ has finite index in $(G)_p$ for each point $p\in M$, then\/ $\Fix_B(M')$ has finite index in\/ $\SingFix_G(M,M')$.
\end{proposition}
\begin{proof}
Let $(G)_{M} = \prod_{p\in M} (G)_p$ and $(B)_{M} = \prod_{p\in M}(B)_p$, and let
\[
\pi\colon \SingFix_G(M,M')\to(G)_M
\]
be the homomorphism that assigns to each element $g\in\SingFix_G(M,M')$ its germs at the points of~$M$.  Then $(B)_M$ has finite index in $(G)_{M}$, and the preimage of $(B)_{M}$ under $\pi$ is precisely $\Fix_B(M')$.
\end{proof}

\begin{proposition}\label{prop:FinitnessPropertiesNormalSubgroups}
Suppose that $(B)_p$ is normal in $(G)_p$ for each $p\in M$, and let \[
\textstyle Q_M = \prod_{p\in M} (G)_p/(B)_p.
\]
If the action of $B$ on\/ $\sing(G)^{|M'|}$ has finitely many orbits, then\/ $\SingFix_G(M,M')$ is an extension of a finite-index subgroup of $Q_M$ by\/ $\Fix_B(M')$.
\end{proposition}
\begin{proof}
Let $\varphi\colon \SingFix_G(M,M')\to Q_M$ be the natural homomorphism, and observe that the kernel of $\varphi$ is precisely~$\Fix_B(M')$.  Our goal is to prove that the image of $\varphi$ has finite index in~$Q_M$.

Note first that each $(G)_p/(B)_p$ can be viewed as a collection of $B$-germs at~$p$, namely those $B$-germs that can be represented by elements $\Fix_G(p)$. Then the elements of $Q_M$ are in one-to-one correspondence to a certain with a certain set $V_M$ of vertices in the germ complex~$K$, consisting of all choice functions
\[
\gamma\colon \sing(G)\to \bigcup_{p\in\sing(G)} \Bgerm^*(G,p)
\]
for which $\gamma(p)\in (G)_p/(B)_p$ for $p\in M$ and $\gamma(p)$ is a trivial $B$-germ for $p\notin M$.  The action of $\SingFix_G(M,M')$ on $V_M$ is the same as the left action of $\SingFix_G(M,M')$ on $Q_M$ induced by~$\varphi$, so it suffices to prove that the action of $\SingFix_G(M,M')$ on $V_M$ has finitely many orbits.

Let $n=|M'|$, and let $w'$ be the vertex of the sublevel complex $K_{\leq n}$ corresponding to the non-singular partial portrait whose set of hidden points is $M'$.  Note that for each $v\in V_M$ there exists a unique $n$-cube $C_v$ that contains both $v$ and $w$, and this lies in $K_{\leq n}$.  By hypothesis, the action of $B$ on $\sing(G)^n$ has finitely many orbits, so Proposition~\ref{prop:FinitelyManyOrbitsCells} tells us that the action of $G$ on~$K_{\leq n}$ has finitely many orbits of cubes.  In particular, the cubes $C_v$ for $v\in V_M$ lie in only finitely many $G$-orbits.  Any element of $G$ mapping $C_v$ to $C_{v'}$ must fix $w$ and map $v$ to $v'$, so the vertices of $V_M$ lie in finitely many orbits under the action of the group
\[
\Stab_G(w) = \{g\in G \mid g(M')=M'\text{ and }\sing(g)\subseteq M'\}.
\]
This group has $\SingFix_G(M',M')$ as a subgroup of finite index, so the vertices of $V_M$ lie in finitely many orbits under the action of $\SingFix(M',M')$.  But any element of $\SingFix(M',M')$ that maps a vertex of $V_M$ to another vertex of $V_M$ must lie in $\SingFix(M,M')$, since any $\gamma\in V_M$ has trivial $B$-germs on $M'-M$.
We conclude that $V_M$ has only finitely many orbits under the action of $\SingFix(M,M')$.\end{proof}

Combining these propositions with Theorem~\ref{thm:SingFixFiniteness} yields the following statement, of which Theorem~\ref{thm:MainFinitenessTheorem} is a special case.

\begin{corollary}\label{cor:MainFinitenessCorollary}
Let $G$ be a finite germ extension of some $B\leq \Homeo(X)$, let $n\geq 1$, and suppose that the following conditions are satisfied:
\begin{enumerate}
    \item The induced action of $B$ on\/ $\sing(G)^n$ has finitely many orbits.\smallskip
    \item $\Fix_B(M)$ has type\/ $\Fn$ for each $M\subseteq \sing(G)$ with\/ $0\leq |M|\leq n$.\smallskip
    \item Either\/ $|(G)_p:(B)_p|<\infty$ for each~$p\in\sing(G)$, or $(B)_p\trianglelefteq (G)_p$ for each $p\in\sing(G)$ and $(G)_p/(B)_p$ has type\/~$\rule{0pt}{11pt}\Fn$.\smallskip
\end{enumerate}
Then $G$ has type\/ $\Fn$.\hfill\qedsymbol
\end{corollary}

Again, note that condition (2) includes the assumption that $B$ has type $\Fn$, since this corresponds to the case where $M=\emptyset$.

\subsection{Application to R\"over--Nekrashevych groups}\label{sec:RoverNekrashevych}

For $d\geq 2$, let $\T_d$ be the infinite, rooted $d$-ary tree, whose vertices are the set $X_d^*$ of finite words in the alphabet $X_d=\{0,\ldots,d-1\}$.  If $f\in \Aut(\T_d)$ and $\alpha\in X_d^*$, the \newword{local action} of $f$ at $\alpha$ is the automorphism $f|_\alpha\in \Aut(\T_d)$ defined by
\[
f(\alpha\beta) = f(\alpha)\,f|_\alpha(\beta).
\]
A group $G\leq \Aut(\T_d)$ is called \newword{self-similar} if $g|_\alpha\in G$ for all $g\in G$ and $\alpha\in X_d^*$.  See \cite{NekBook} for a general introduction to self-similar groups.  We will need two specific classes of self-similar groups:
\begin{itemize}
    \item A self-similar group $G\leq\Aut(\T_d)$ is \newword{contracting} if there exists a finite set $\mathcal{N}\subset G$ such that for each $g\in G$, we have $g|_\alpha\in\mathcal{N}$ for all but finitely many $\alpha\in X_d^*$.  The minimal such $\mathcal{N}$ is the \newword{nucleus} for~$G$.\smallskip
    \item A self similar group $G\leq \Aut(\T_d)$ is a \newword{bounded automata group} if it is finitely generated and for each $g\in G$, there exist only finitely many infinite words $i_1i_2i_3\cdots \in X_d^\omega$ for which all of the local actions $f|_{i_1\cdots i_n}$ are nontrivial.
\end{itemize}
Bondarenko proved that every bounded automata group is contracting~\cite{Bond}.

Given a $d\geq 2$, $r\geq 1$, and a self-similar group $G\leq\Aut(\T_d)$, the corresponding \newword{R\"over--Nekrashevych group} $\boldsymbol{V_{d,r}G}$ is defined as follows.  Let $\C_{d,r}$ denote the Cantor space $X_r\times X_d^\omega$.  For a finite prefix $\alpha\in X_r\times X_d^*$, the corresponding \newword{cone} in $\C_{d,r}$ is the set of all sequences that begin with~$\alpha$.  Then a homeomorphism $f\in \Homeo(\C_{d,r})$ lies in $V_{d,r}G$ if and only if there exist finite partitions $C_{\alpha_1},\ldots,C_{\alpha_n}$ and $C_{\beta_1},\ldots,C_{\beta_n}$ of $\C_{d,r}$ into cones and elements $g_1,\ldots,g_n\in G$ such that 
\[
f(\alpha_i\psi) = \beta_i\,g_i(\psi)
\]
for all $1\leq i\leq n$ and all $\psi \in X_d^\omega$.  In the special case where $G$ is trivial, the resulting group $V_{d,r}G$ is the \newword{Higman--Thompson group} $V_{d,r}$ (see~\cite{Hig2}).

The class of R\"over--Nekrashevych groups was first considered by Scott~\cite{Scott}.  R\"over made the connection to groups such as Grigorchuk's group~\cite{Rov}, and Nekrashevych gave the modern definition using self-similar groups~\cite{Nek}.  Note that Nekrashevych only defined the groups $V_dG = V_{d,1}G$, so our definition here is slightly more general.  As with the Higman--Thompson groups, which are often non-isomorphic for different values of $r$~\cite{Pardo}, there is no reason to expect $V_{d,r}G$ and $V_{d,s}G$ to be isomorphic when $r\not\equiv s\;(\mathrm{mod}\;d-1)$.

\begin{proposition}\label{prop:BoundedAutomataGermExtension}
Let $G\leq \mathrm{Aut}(\mathcal{T}_d)$ be a bounded automata group. Then each R\"over--Nekrashevych group $V_{d,r}G$ is a finite germ extension of the corresponding Higman--Thompson group~$V_{d,r}$.
\end{proposition}
\begin{proof}
Let $f\in V_{d,r}G$.  Then there exist partitions $C_{\alpha_1},\ldots,C_{\alpha_n}$ and $C_{\beta_1},\ldots,C_{\beta_n}$ of $\C_{d,r}$ into cones and elements $g_1,\ldots,g_n\in G$ such that $f(\alpha_i\psi)=\beta_i\,g_i(\psi)$ for all $\psi\in X_d^\omega$. Since $G$ is a bounded automata group, each $\sing(g_i)$ (with respect to the base group $V_d$) is a finite set, so
\[
\sing(f) = \{\alpha_i\psi \mid 1\leq i\leq n \text{ and }\psi\in \sing(g_i)\}
\]
is finite as well.  Furthermore, given any singular point $\alpha_i\psi$ for $f$, let $\gamma\in X_d^*$ be a prefix of $\psi$ so that $\psi$ is the only singular point of $g_i$ that has $\gamma$ as a prefix.  Then $g_i$ maps the cone $C_\gamma$ to some cone $C_\delta$, so $f$ maps $C_{\alpha_i\gamma}$ to $C_{\beta_i\delta}$.  Let $f'$ be the element of $V_{d,r}G$ that maps $C_{\alpha_i\gamma}$ to $C_{\beta_i\delta}$ via $f$, maps $C_{\beta_i\delta}$ to $C_{\alpha_i\gamma}$ by  prefix replacement, and is the identity elsewhere. Then $f'$ agrees with $f$ on the neighborhood $C_{\alpha_i\gamma}$ of $\psi$ and satisfies $\sing(f')=\{\psi\}$, as desired.
\end{proof}

The following is a slightly generalized version of Theorem~\ref{thm:NekrashevychGroups}.

\begin{theorem}\label{thm:RoverNek}
Let $d\geq 2$, and let $G\leq \Aut(\T_d)$ be a bounded automata group.  Then each of the corresponding R\"over--Nekrashevych groups $V_{d,r}G$ has type\/~$\Finfty$.
\end{theorem}
\begin{proof}
Recall that a point $p\in\C_{d,r}$ is \newword{rational} if it is eventually repeating, i.e.\ if $p=\alpha\overline{\beta}$ for some $\alpha\in X_r\times X_d^*$ and $\beta\in X_d^+$.  Since R\"over--Nekrashevych groups are rational similarity groups (see Remark~\ref{thm:RSGs}),
it follows from \cite[Proposition~5.5]{BBMZ2} that the group of germs $(V_{d,r}G)_p$ is virtually cyclic for any rational point $p$, and hence $|(V_{d,r}G)_p:(V_{d,r})_p|<\infty$ for all $p\in\sing(V_{d,r}G)$.  Brown~\cite{Bro1} proved that the Higman--Thompson groups $V_{d,r}$ have type $\Finfty$, and it is proven in Appendix~\ref{sec:AppendixA} that the stabilizer in $V_{d,r}$ of any finite set of rational points has type~$\Finfty$.  Finally, it is well known that $V_{d,r}$ acts with finitely many orbits on $O^n$ for any $V_{d,r}$-orbit $O\subset \C_{d,r}$ and any $n\geq 1$.  Since $\sing(V_{d,r}G)$ is a finite union of such orbits,  it follows that $V_{d,r}$ acts with finitely many orbits on $\sing(V_{d,r}G)^n$ for all $n\geq 1$, so by Theorem~\ref{thm:MainFinitenessTheorem} it follows that $V_{d,r}G$ has type~$\Finfty$.
\end{proof}

\begin{remark}
\label{thm:RSGs}
In \cite{BBMZ2}, the first author, Collin Bleak, the third author, and Matthew Zaremsky introduced the class of \newword{rational similarity groups} (RSG's), which generalize R\"over--Nekrashevych groups to allow for asynchronous automata over irreducible subshifts of finite type.  The class of R\"over--Nekrashevych groups over bounded automata groups generalizes in an obvious way to a class of \newword{bounded RSG's}, and the proof of Theorem~\ref{thm:RoverNek} generalizes to show that any full, bounded, contracting RSG has type~$\Finfty$.
\end{remark}

\begin{remark} We recall that, for a given automorphism $g$ of the infinite tree $\T_d$, its \newword{activity growth function} 
$k\mapsto \theta_k(g)$ counts the number non-trivial local actions among words of
$X_d^*$ of length $k$ (see \cite{Sidki}).
We say that $g$ has \newword{polynomial activity growth} of degree at most~$n$ if this function is bounded above by some polynomial of degree~$n$. A self-similar group $G$ has polynomial activity growth if there exists an $n$ so that all elements of $G$ have polynomial activity growth of degree at most~$n$.

Theorem~\ref{thm:RoverNek} can be extended to contracting self-similar groups with polynomial activity growth.
In particular, for a contracting self-similar group $G$ whose elements have polynomial activity growth of degree at most~$n$, we can form a chain of contracting self-similar subgroups $G=G_n\geq G_{n-1}\geq \cdots \geq G_1\geq G_0$, where $G_i$ is the subgroup of all elements of $G$ whose activity growth has degree at most~$i$.  Note that $G_0$ is a bounded automata group, so $V_{d,r}G_0$ has type~$\Finfty$ by Theorem~\ref{thm:RoverNek}.  From the structure of the automata, one can see that each $V_{d,r}G_{i+1}$ is a finite germ extension of $V_{d,r}G_i$, and from arguments similar to those in Appendix~\ref{sec:AppendixA} and Theorem~\ref{thm:RoverNek}, it follows by induction that each $V_{d,r}G_i$ has type~$\Finfty$.
\end{remark}

\section{Abelianizations and Simplicity}
\label{sec:abelianization-simplicity}

In this section we prove Theorems~\ref{thm:Simplicity} and \ref{thm:Abelianization} from the introduction regarding the abelianization of a finite germ extension and the simplicity of its commutator subgroup.  In both cases we prove a slight refinement of the version from the introduction (see Theorems~\ref{thm:ActualVersion} and \ref{thm:SimplicityRefined} below).  As an application, in Section~\ref{subsec:PDiff} we prove Theorem~\ref{thm:PiecewiseCr} from the introduction on piecewise diffeomorphism groups.  Finally, we discuss some further applications of our theory in Section~\ref{subsec:Examples}.

\subsection{Abelianization}
\label{ssec:abelianization}

As described in the introduction, if $G$ is a finite germ extension of some $B\leq\Homeo(X)$ and $p\in X$, let $A_p(G)$ be the abelian group that fits into the exact sequence
\[
(B)_p/(B)_p' \longrightarrow (G)_p/(G)_p' \longrightarrow A_p(G) \longrightarrow 0.
\]
That is, $A_p(G) = (G)_p/N$, where $N=(G)_p'(B)_p$ is the (normal) subgroup of $(G)_p$ generated by $(G)_p'$ and~$(B)_p$.

Let $\nc{B}_G$ denote the normal closure of $B$ inside of $G$.  Our main goal in this section is to prove the following theorem.

\begin{theorem}\label{thm:ActualVersion}
Let $G$ be a finite germ extension of some $B\leq \Homeo(X)$, and suppose that $B$ has no global fixed points, and that $B$ and $G$ have the same orbits in $X$. Let $R$ be a system of representatives for the $B$-orbits in~$X$.  Then there exists an epimorphism
\[
\sigma\colon G \to \bigoplus_{p\in R} A_p(G)
\]
such that\/ $\ker(\sigma)=\nc{B}_G=G'B$. Furthermore, for each $p\in R$, the corresponding component $\sigma_p\colon G\to A_p(G)$ of $\sigma$ is the unique homomorphism with the following properties:
\begin{enumerate}
    \item If $g\in G$ has no singular points in\/ $\Orb(p)\coloneqq \Orb_G(p)=\Orb_B(p)$,
    then $\sigma_p(g)=0$.\smallskip
    \item If $g\in G$ fixes $p$ and\/ $\sing(g)\cap \Orb(p)=\{p\}$, then $\sigma_p(g)$ is the image of\/~$(g)_p$ in $A_p(G)$.
\end{enumerate}
\end{theorem}

We refer to the epimorphism $\sigma$ defined by this theorem as the \newword{germ abelianization homomorphism}. Since $\sigma$ is surjective and $\ker(\sigma)=G'B$, this theorem gives an exact sequence
\[
B/B' \longrightarrow G/G' \overset{\tilde{\sigma}}{\longrightarrow} \bigoplus_{p\in R} A_p(G) \longrightarrow 0
\]
where $\tilde{\sigma}$ is the homomomorphism induced by~$\sigma$.  In particular, Theorem~\ref{thm:ActualVersion} immediately implies Theorem~\ref{thm:Abelianization}.  In addition, Theorem~\ref{thm:ActualVersion} explicitly describes~$\sigma$, and includes the additional assertion that $\nc{B}_G=G'B$, which we will need for the proof of Theorem~\ref{thm:Simplicity}.

\begin{remark}
\label{thm:when-B-orbits-dont-coincide-G-orbits}
Theorem~\ref{thm:ActualVersion} includes the hypothesis that $B$ and $G$ have the same orbits in~$X$.  If $G$ does not have this property, then $G$ has a subgroup $G_0$ of elements that do preserve the $B$-orbits, and each of the inclusions $B\leq G_0$ and $G_0\leq G$ is a finite germ extension.  If $\mathrm{Orbs}$ is the set of $B$-orbits in $X$ and $\mathbb{Z}\mathrm{Orbs}$ is the free abelian group generated by $\mathrm{Orbs}$, then there is a homomorphism $\tau\colon G\to \mathbb{Z}\mathrm{Orbs}$ defined by
\[
\tau(g) = \sum_{x\in \sing(g)} \bigl(\mathrm{Orb}_B(gx)-\mathrm{Orb}_B(x)\bigr)
\]
whose kernel contains $G_0$. 
 This homomorphism is always nontrivial if the $B$-orbits are different from the \mbox{$G$-orbits}, and hence $G$ is never simple in this case.
\end{remark}

The proof of Theorem~\ref{thm:ActualVersion} occupies the remainder of this section.

\begin{notation}For the rest of this section, let $G$ be a finite germ extension of some~$B\leq \Homeo(X)$, and let $R$ be a system of representatives for the $B$-orbits in $X$. We assume that $B$ has no global fixed points, and $B$ and $G$ have the same orbits in~$X$. We will write $\Orb(p)$ for the set $\Orb_G(p)=\Orb_B(p)$. By condition~(3) in the definition of a finite germ extension, it follows that $G$ is generated by the elements of $B$ together with all \newword{elementary homeomorphisms} in~$G$, where a homeomorphism $g$ is elementary if it has a unique singular point and this is a fixed point of~$G$.
\end{notation}

We begin by proving the existence of the component homomorphisms $\sigma_p$ of the germ abelianization homomorphism.

\begin{proposition}\label{prop:DefinitionSigmap}
For each point $p\in X$, there exists a unique homomorphism $\sigma_p\colon G\to A_p(G)$ with the following properties:
\begin{enumerate}
    \item
    If $g\in G$ has no singular points in\/ $\mathrm{Orb}(p)$, then $\sigma_p(g)=0$.\smallskip
    \item If $g\in G$ fixes $p$ and\/ $\sing(g)\cap \mathrm{Orb}(p)=\{p\}$, then $\sigma_p(g)$ is the image of $(g)_p$ in~$A_p(G)$.
\end{enumerate}
\end{proposition}
\begin{proof}
If $g\in G$ and $g(p)=p$, let $[g]_p$ denote the image of $(g)_p$ in $A_p(G)$.  Note then that $[gh]_p=[g]_p+[h]_p$ for any $g,h\in G$ with $g(p)=h(p)=p$.

For each point $q\in \mathrm{Orb}(p)$, fix an element $b_q\in B$ for which $b_q(p)=q$, where we choose $b_p=1$. Define a function $\sigma_p\colon G\to A_p(G)$ by
\[
\sigma_p(g) = \sum_{q\in \mathrm{Orb}(p)} \bigl[b_{g(q)}^{-1}\,g\,b_q\bigr]_p.
\]
Note that $\bigl[b_{g(q)}^{-1}\,g\,b_q\bigr]_p$ is only nonzero if $q$ is a singular point for~$g$, and therefore the given sum has only finitely many nonzero terms.  Furthermore, this function clearly satisfies properties (1) and (2) above.  It is also a homomorphism, since
if $g,h\in G$, then
\begin{multline*}
\sigma_p(gh) = \sum_{q\in \mathrm{Orb}(p)} \bigl[b_{gh(q)}^{-1}\,gh\,b_q\bigr] = \sum_{q\in \mathrm{Orb}(p)} \Bigl(\bigl[b_{gh(q)}^{-1}\,g\,b_{h(q)}\bigr] + \bigl[b_{h(q)}^{-1}\,h\,b_q\bigr]\Bigr) \\[10pt]
= \sum_{r\in \mathrm{Orb}(p)} \bigl[b_{g(r)}^{-1}\,g\,b_{r}\bigr] + \sum_{q\in \mathrm{Orb}(p)} \bigl[b_{h(q)}^{-1}\,h\,b_q\bigr] = \sigma_p(g)+\sigma_p(h)
\end{multline*}
where $r=h(q)$.

Finally, the uniqueness of $\sigma_p$ follows from the fact  that $G$ is generated by the elements $g\in G$ that satisfy either condition (1) or condition (2) above.  In particular, any element of $B$ satisfies condition~(1), as does any elementary homeomorphism whose singular point does not lie in~$\mathrm{Orb}(p)$. Any elementary homeomorphism whose singular point lies in $\mathrm{Orb}(p)$ is conjugate by an element of $B$ to an elementary homeomorphism with singular point~$p$, and these satisfy condition~(2), so the elements satisfying conditions (1) or (2) generate~$G$.
\end{proof}

\begin{remark}
If $\mathcal{O}$ is an orbit in $X$, then the homomorphism $\sigma_p$ also does not depend on the chosen point $p\in\mathcal{O}$, in the following sense.  If we pick two points $p,q\in\mathcal{O}$, then there exists a natural isomorphism $\varphi_{pq}\colon A_p(G)\to A_q(G)$ defined by
\[
\varphi_{pq}\bigl([g]_p\bigr) = \bigl[bgb^{-1}\bigr]_q
\]
where $b$ is any element of $B$ for which $b(p)=q$, and $[g]_p$ denotes the image of $(g)_p$ in $A_p(G)$.  It is easy to check that~$\varphi_{pq}$ does not depend on the chosen element~$b$, and furthermore $\sigma_q = \varphi_{pq}\sigma_p$.
\end{remark}

\begin{proposition}\label{prop:DefinitionSigma}
There exists an epimorphism $\sigma\colon G\to \bigoplus_{p\in R}A_p(G)$ with components~$\sigma_p$.
\end{proposition}
\begin{proof}
Note first that $\sigma_p(g)$ is only nonzero if $g$ has a singular point in $\Orb(p)$. 
In particular, only finitely many $\sigma_p(g)$ can be nonzero for a given $g\in G$, so the image of $\sigma$ does lie in the direct sum $\bigoplus_{p\in R} A_p(G)$.

Next, observe that any elementary homeomorphism $g$ with singular point~$p\in R$ maps to its image in $A_p(G)$ under $\sigma_p$ and maps to $0$ under $\sigma_{p'}$ for all $p'\ne p$.  Thus the image of $\sigma$ contains a generating set for $\bigoplus_{p\in R} A_p(G)$, so $\sigma$ is surjective. 
\end{proof}

\begin{lemma}\label{lem:ElementaryCommutator}
If $f,g\in G$ are elementary homeomorphisms with the same singular point $p$, then $[f,g]\in \nc{B}_G$.
\end{lemma}
\begin{proof}
%Since $p$ is not a global fixed point of $B$ there is $b\in B$ with $b(p) \neq p$.
%Since $b$ is a homeomorphism there is a neighbourhood $P$ of $p$ with $P \cap b(P) = \varnothing$.
%By the definition of finite germ extension there is $f' \in Bf$ and $g' \in Bg$ with $\supp(f') \subseteq P$ and $\supp(g') \subseteq P$.
%Now $f'$ commutes with $g'^b$ so $f/\nc{B} = f'/\nc{B}$ and $g/\nc{B} = g'/\nc{B}$ commute so $[f,g] \in \nc{B}_G$.
%--------------
We will proceed by the following steps:
\begin{enumerate}
    \item Find an element $f'\in BfB$ which is the identity in a neighborhood of~$p$.  Then $h= g^{-1}f'g$ is the identity in a neighborhood of $p$, and has only one singular point.\smallskip
    \item Find an element $h'\in BhB$ so that $(h')_p=(f)_p$.\smallskip
    \item Deduce that $f^{-1}h'$ has no singular points and therefore lies in $B$.
\end{enumerate}
Then $\varphi(f^{-1}h')=1$, where $\varphi\colon G\to G/\nc{B}_G$ is the quotient homomorphism.  Since $\varphi(f)=\varphi(f')$ and $\varphi(h)=\varphi(h')$, we get
\[
\varphi([f,g]) = \varphi(f^{-1}g^{-1}fg) = \varphi(f^{-1}g^{-1}f'g) = \varphi(f^{-1}h) = \varphi(f^{-1}h') = 1
\]
as desired.

For step (1), we choose any $b\in B$ for which $b(p)\ne p$.  Then the only singular point of $fb$ is at $b^{-1}(p)$, which is not equal to~$p$. Therefore, there exists a $c\in B$ so that $(fb)_p = (c)_p$.  Then $f'=c^{-1}fb$ is the identity in a neighborhood of~$p$, and it follows that $h=g^{-1}f'g$ is the identity in a neighborhood of $p$, and has only one singular point.

For step (2), observe that $c^{-1}(p) \ne p$ since $c(p) = fb(p)\ne p$.  Since $b^{-1}(p)\ne p$ and $\sing(g^{-1})=\{p\}$, it follows that $p$ is not a singular point of either $g^{-1}b^{-1}$ or $g^{-1}c^{-1}$, so there exist $d,e\in B$ such that $(d)_p=(g^{-1}b^{-1})_p$ and $(e)_p=(g^{-1}c^{-1})_p$.  Then $(bgd)_p=(cge)_p=(\mathrm{id})_p$, so
\[
(e^{-1}hd)_p = (e^{-1}g^{-1}c^{-1}fbgd)_p =  (\mathrm{id})_p^{-1}(f)_p(\mathrm{id})_p = (f)_p
\]
and hence $h'=e^{-1}hd$ has the desired property.

For step (3), since $h$ has only one singular point, we know that $h'$ has only one singular point, which must be the point~$p$. Then $f^{-1}h'$ has no singular points and therefore lies in $B$, as desired.
\end{proof}

\begin{proof}[Proof of Theorem~\ref{thm:ActualVersion}]
All that remains is to prove that $\ker(\sigma)=\nc{B}_G$, and that this is equal to $G'B$.  Clearly $\nc{B}_G \leq \ker(\sigma)$.  For the opposite inclusion, let $f\in\ker(\sigma)$.  We must show that $f\in \nc{B}_G$.  We proceed by induction on $|\sing(f)|$.  The base case is $|\sing(f)|=0$, for which $f\in B$ and hence $f\in\nc{B}_G$.  Suppose then that $|\sing(f)|\geq 1$, and let $p\in R$ so that $\Orb(p)$ contains a singular point for~$f$.

Suppose first that $\Orb(p)$ contains exactly one singular point for~$f$.  Left and right-multiplying by elements of~$B$, we may assume that $p$ is a singular point for $f$ and $f(p)=p$.  Then $\sigma_p(f)$ is the image of $(f)_p$ in $A_p(G)$.  Since $\sigma_p(f)=0$, it follows that  $(f)_p \in (B)_p\,(G)_p'$.  Indeed, left-multiplying by an element of $B$ that fixes $p$, we may assume that $(f)_p\in (G)_p'$.  Then
\[
(f)_p = [(g_1)_p,(h_1)_p]\cdots [(g_n)_p,(h_n)_p]
\]
for some elements $(g_1)_p,\ldots,(g_n)_p,(h_1)_p,\ldots,(h_n)_p\in (G)_p$.  Since every element of $(G)_p$ can be represented by either an element of $B$ that fixes $p$ or an elementary homeomorphism with singular point $p$, we may assume that each $g_i$ and $h_i$ is one of these two types.  If both $g_i$ and $h_i$ are elementary, then since $p$ is not a global fixed point for~$B$, Lemma~\ref{lem:ElementaryCommutator} tells us that $[g_i,h_i]\in \nc{B}_G$.  Clearly $[g_i,h_i]\in \nc{B}_G$ if either $g_i$ or $h_i$ lies in~$B$, so we conclude that the product $k=[g_1,h_1]\cdots[g_n,h_n]$ lies in $\nc{B}_G$.  But $k$ is an elementary homeomorphism with $(k)_p=(f)_p$, so $k^{-1}f$ has fewer singular points than~$f$.  By our induction hypothesis, we conclude that $k^{-1}f\in\nc{B}_G$, and therefore $f\in \nc{B}_G$.

Finally, suppose that $\mathrm{Orb}(p)$ has at least two singular points for~$f$.  Left and right-multiplying by elements of $B$, we may assume that one of these singular points is~$p$, and that $f(p)=p$.  Let $h\in G$ be an elementary homeomorphism such that $(h)_p=(f)_p$.  Let $q$ be another singular point for $f$ in $\mathrm{Orb}(p)$, let $b\in B$ so that $b(p)=h^{-1}f(q)$, and consider the commutator  $[b,h]=bhb^{-1}h^{-1}$.  Note that $bhb^{-1}$ is an elementary homeomorphism with singular point $b(p) = h^{-1}f(q)$, so
\[
\sing([b,h]) \subseteq h\, \sing(bhb^{-1})\cup \sing(h^{-1}) = h\{h^{-1}f(q)\}\cup\{p\} = \{f(q),p\}
\]
(See Remark~\ref{rem:SingularPointsComposition}).  Therefore,
\[
\sing([b,h]f)\subseteq f^{-1}\sing([b,h])\cup\sing(f) \subseteq \{p,q\}\cup \sing(f) = \sing(f).
\]
But $(h^{-1}f)_p=(h)_p^{-1}(f)_p=(\mathrm{id})_p$, and $p$ is not a singular point for $bhb^{-1}$ since $h^{-1}f(q)\ne p$, so $p$ is not a singular point for $[b,h]f = (bhb^{-1})(h^{-1}f)$.  We conclude that $[b,h]f$ has fewer singular points than~$f$, so $[b,h]f\in \nc{B}_G$ by our inductive hypothesis.  Since $[b,h]\in \nc{B}_G$, it follows that $f\in \nc{B}_G$.

This proves that $\ker(\sigma)=\nc{B}_G$.  To prove that $\nc{B}_G=G'B$, let $\pi\colon G\to G/G'$ be the abelianization homomorphism.  Then the image of $\nc{B}_G$ under $\pi$ is the normal closure of $\pi(B)$ in $G/G'$, which is equal to $\pi(B)$ since $G/G'$ is abelian.  But $G'\leq \ker(\sigma)=\nc{B}_G$ since $\bigoplus_{p\in R}A_p(G)$ is abelian, so $\nc{B}_G = \pi^{-1}(\pi(\nc{B}_G)) = \pi^{-1}(\pi(B)) = G'B$.
\end{proof}

\subsection{Simplicity}
\label{ssec:simplicity}

Our goal in this section is to prove the following theorem, which is a slight refinement of Theorem~\ref{thm:Simplicity} from the introduction.

\begin{theorem}\label{thm:SimplicityRefined}
Let $G$ be a finite germ extension of some $B\leq\Homeo(X)$. Suppose that $B$ is simple, locally moving, and has no global fixed points, and that $B$ and $G$ have the same orbits in~$X$. Then every nontrivial subgroup of\/ $G$ which is normalized by\/ $G'$ contains\/~$G'$.  In particular, $G'$ is simple, and every proper quotient of $G$ is abelian. \end{theorem}

Here the group $B$ is \newword{locally moving} if for each nonempty open set $U\subseteq X$, there exists a nontrivial element of $B$ which is supported on~$U$.  Note that any locally moving group $B$ of homeomorphisms of a Hausdorff space must be nonabelian, for if $b$ is any nontrivial element of $B$, then there must exist a nonempty open set $U$ so that $b(U)\cap U=\emptyset$, in which case $b$ cannot commute with any element of $B$ which is supported on $U$.

\begin{lemma}\label{lem:NormalClosure}
Let $G$ be a group, and let $B\leq G'$.  Suppose that for every $g\in G$ there exists $c\in B$ so that $\nc{c}_B = B$ and $c^g\in B$. Then $\langle\!\langle B \rangle\!\rangle_G = \langle\!\langle B\rangle\!\rangle_{G'}$.
\end{lemma}
\begin{proof}
Clearly $\nc{B}_{G'}\leq \nc{B}_G$.  For the opposite inclusion, it suffices to prove that $B^g\leq \nc{B}_{G'}$ for all $g\in G$.

Let $g\in G$, and let $c\in B$ so that $\nc{c}_B=B$ and $c^g\in B$.  Then $B$ is generated by the conjugates $\{c^b\mid b\in B\}$, so $B^g$ is generated by $\{c^{bg}\mid b\in B\}$.  But
$c^{bg} = (c^{gb})^{[b,g]}$ for any $b\in B$, and $c^{gb}\in B$ since $c^g$ and $b$ both lie in~$B$, which means that $c^{bg}\in \nc{B}_{G'}$. 
 We conclude that $B^g\leq \nc{B}_{G'}$ for all $g\in G$, and therefore $\nc{B}_G\leq \nc{B}_{G'}$.
\end{proof}

\begin{proposition}\label{prop:NormalClosuresEqual}
Let $G$ be a finite germ extension of some $B\leq \Homeo(X)$, and suppose $X$ is infinite and $B$ is locally moving and simple. Then $B\leq G'$ and $\nc{B}_{G'}=\nc{B}_G$.
\end{proposition}
\begin{proof}
We know that $B$ is nonabelian since it is locally moving.  Since $B$ is simple, it follows that $B=B'$, and hence $B\leq G'$.

To prove that $\nc{B}_{G'}=\nc{B}_G$, let $g\in G$.  Since $g$ has finitely many singular points and $X$ is infinite, there exists a nonempty open set $U$ that has no singular points of~$g$.  Since $B$ is locally moving, there exists a nontrivial $c\in B$ which is supported on~$U$.  Then $\nc{c}_B=B$ since $B$ is simple, and $c^g\in B$ since $c^g$ has no singular points.  Since $g$ was arbitrary, it follows from Lemma~\ref{lem:NormalClosure} that $\langle\!\langle B\rangle\!\rangle_{G}=\langle\!\langle B\rangle\!\rangle_{G'}$.
\end{proof}

\begin{proof}[Proof of Theorem~\ref{thm:SimplicityRefined}]
Let $N$ be a nontrivial subgroup of $G$ which is normalized by~$G'$. Note again that $B$ must be nonabelian since it is locally moving. Since $B$ is simple, it follows that $B=B'$, so $B\leq G'$ and hence $B$ normalizes~$N$.

We claim first that $N$ contains $B$.  To prove this, let $n$ be any nontrivial element of~$N$. Since $X$ is Hausdorff, we can find a nonempty open set $U$ in $X$ such that $n(U)$ is disjoint from~$U$.  Since $B$ is locally moving, we know that $X$ has no isolated points, and since $U$ is Hausdorff it follows that $U$ is infinite.  Since $n$ has only finitely many singular points we can find a non-singular point $q$ for $n$ in $U$.  Then there exists a neighborhood $V$ of $q$ with $V\subseteq U$ and an element $b_1\in B$ so that $n$ agrees with $b_1$ on~$V$.  Let $b_2$ be a nontrivial element of $B$ that is supported on~$V$.  Then $nb_2n^{-1} = b_1b_2b_1^{-1}$ is supported on~$n(V)$, which is disjoint from~$V$, so $nb_2n^{-1}b_2^{-1}=b_1b_2b_1^{-1}b_2^{-1}$ is a nontrivial element of~$B$ that is contained in~$N$.  Then $N\cap B\ne 1$, and since $B$ is simple and $B$ normalizes $N$ it follows that $B\leq N$.

Since $N$ is normalized by $G'$, it follows that $N$ contains $\nc{B}_{G'}$.  By Proposition~\ref{prop:NormalClosuresEqual}, we know that 
$\nc{B}_{G'} = \nc{B}_G$, and therefore $N$ contains $\nc{B}_G$.  Since $B$ has no global fixed points and $B$ and $G$ have the same orbits in $X$, Theorem~\ref{thm:ActualVersion} tells us that $\nc{B}_G = G'B$, so in particular $\nc{B}_G$ contains~$G'$, and therefore $N$ contains~$G'$.
\end{proof}

\subsection{Application to piecewise-diffeomorphism groups}\label{subsec:PDiff}

In this section we prove the following theorem from the introduction.

{
\renewcommand{\thetheorem}{\ref{thm:PiecewiseCr}}
\begin{theorem}\textPiecewiseCr
\end{theorem}
\addtocounter{theorem}{-1}
}

Here $\PDiff_0^r(X)$ is the group of compactly supported, orientation-preserving piecewise-$C^r$-diffeomorphic homeomorphisms of $X$, which is a finite germ extension of the group $\Diff_0^r(X)$ of compactly supported $C^r$ diffeomorphisms of $X$. It is known that the commutator subgroup $\Diff_0^r(X)'$ is simple for all $1\leq r\leq \infty$ \cite{Eps}, and for $r\ne 2$ the group $\Diff_0^r(X)$ is perfect and hence simple\footnote{See \cite{Eps} for simplicity of the commutator subgroup. 
 For $\Diff_0^r(X)$ being perfect, see \cite{Thurston} for the $r=\infty$ case, \cite{Mather} for the $3\leq r<\infty$ case, and \cite{Mather2} for the $r=1$ case.  It is an open problem whether $\Diff_0^r(X)$ is perfect for $r=2$.}

%\begin{lemma}
%Let $X$ be a connected $1$-manifold, let $r=1$ or $3\leq r\leq \infty$, and let $B=\Diff_0^r(X)$.  Then for any point $p\in X$, the commutator subgroup $(B)_p'$ is the group of all germs $(f)_p\in (B)_p$ for which $f'(p)=1$.
%\end{lemma}
%\begin{proof}
%Let $H$ be the subgroup of elements of $(B)_p$ for which $f'(p)=1$. Then $H$ is the kernel of the homomorphism $(B)_p\to \R$ that maps $(f)_p$ to $\log f'(p)$, and therefore $(B)_p'\leq H$.  For the opposite inclusion, fix a local coordinate $x$ in a neighborhood of~$p$ with $x(p)=0$, let $(g)_p$ be the germ for $g(x)=2x$, and let $(h)_p\in H$.  It suffices to prove that there exists an element $(b)_p\in (B)_p$ so that $hb=g^{-1}bg$.  Use $b(x)=x$ for $x\in[1,2]$, and in general set
%\[
%b(x) = (gh)^{-n}g^n(x)
%\]
%for $x\in [2^{-n},2^{1-n}]$. \blue{This lemma would be done if we could show that $b$ is $C^r$ at $0$.  (Note that it's already $C^r$ at $1$ and indeed $2^{-n}$ for all $n$.)  This would be a much better proof than the one we have.}
%\end{proof}

\begin{lemma}\label{lem:JetCommutators}
Let $J_\infty$ be the group of all formal power series\/ $\sum_{i=1}^\infty a_ix^i$ with $a_i\in\R$ and $a_1>0$ under the operation of formal composition.  Then the commutator subgroup $J_\infty'$ consists of all elements of $J_\infty$ for which $a_1=1$.
\end{lemma}
\begin{proof}
Let $H$ be the subgroup of $J_\infty$ consisting of elements for which $a_1=1$.  It follows easily from the chain rule that $J_\infty'\leq H$.  For the opposite inclusion, let $g(x)=2x$ and let $h(x) = x+\sum_{i=2}^\infty a_ix^i$ be any element of $H$.  It suffices to prove that there exists an element $k(x) = x + \sum_{i=2}^\infty b_ix^i$ in $H$ so that $k\circ h = g^{-1}\circ k\circ g$.  To see this, observe that
\[
(k\circ h)(x) = x + \sum_{i=2}^\infty \bigl(b_i + p_i(a_2,\ldots,a_i,b_2,\ldots,b_{i-1})\bigr)x^i
\]
where each $p_i$ is a polynomial.  Since
\[
(g^{-1}\circ k\circ g)(x) = x + \sum_{i=2}^\infty 2^{i-1}b_ix^i
\]
we can recursively define
\[
b_i = \frac{p_i(a_2,\ldots,a_i,b_2,\ldots,b_{i-1})}{2^{i-1}-1}
\]
for each $i\geq 2$ to get the desired element~$k$.
\end{proof}

\begin{remark}\label{rem:jets} The formal power series in Lemma~\ref{lem:JetCommutators} are known in differential geometry as \newword{$\boldsymbol{\infty}$-jets} (in one dimension), and $J_\infty$ is the corresponding \newword{jet group}.  For $1\leq r<\infty$, there is also a group $J_r$ of \newword{$\boldsymbol{r}$-jets}, which is the quotient of $J_\infty$ obtained by ignoring terms of degree higher than~$r$.  That is, $J_r$ is the group of all degree-$r$ Maclaurin polynomials for an increasing $C^r$-diffeomorphism of $\R$ that fixes~$0$, under the natural operation induced by composition. See \cite[Chapter~IV]{KMS} for more information on jets and jet groups.
\end{remark}

\begin{lemma}\label{lem:WhatIsNHere}
Let $X$ be a connected $1$-manifold. Let $1\leq r\leq \infty$, let $B=\Diff_0^r(X)$, and let $G=\PDiff_0^r(X)$.  Then for any $p\in X$, the group $A_p(G)$ is isomorphic to\/ $\mathbb{R}$, 
with the corresponding quotient homomorphism $(G)_p\twoheadrightarrow A_p(G)\xrightarrow{\sim}\mathbb{R}$ being
\[
(f)_p \mapsto \log\biggl(\frac{f^+(p)}{f^-(p)}\biggr).
\]
\end{lemma}
\begin{proof}Recall that $A_p(G)=(G)_p/(G)_p'(B)_p$.  The given mapping $\mu\colon (G)_p\to \R$ is clearly an epimorphism, and its kernel is the group $(G)_p^{\text{Diff}}$ of all germs $(f)_p\in (G)_p$ for which $f$ is differentiable at $p$.  By the first isomorphism theorem, it suffices to prove that $(G)_p^{\text{Diff}} = (G)_p'(B)_p$.

The inclusion $(G)_p'(B)_p\leq (G)_p^{\text{Diff}}$ is not difficult, since both $(G)_p'$ and $(B)_p$ are contained in $(G)_p^{\text{Diff}}$.  In particular, it follows from the chain rule that any element of $(G)_p'$ has left and right slopes equal to $1$ at $p$, which proves that $(G)_p'\leq (G)_p^{\text{Diff}}$, and $(B)_p\leq (G)_p^{\text{Diff}}$ since elements of $(B)_p$ are $C^r$ at~$p$.

For the opposite inclusion, let $(f)_p\in (G)_p^{\text{Diff}}$, and consider the homomorphism $\psi\colon (G)_p\to J_r\times J_r$ that assigns to each germ its left and right $r$-jets at~$p$.  This homomorphism $\psi$ is surjective---this is clear when $r<\infty$, and for $r=\infty$ it follows from a theorem of Borel (cf.\ \cite[1.3]{MoRe}). It follows that $\psi$ maps $(G)_p'$ onto the commutator subgroup $(J_r\times J_r)'=J_r'\times J_r'$, which by Lemma~\ref{lem:JetCommutators} is the set of all jet pairs $\bigl(\sum_{i=1}^r a_ix^i,\sum_{i=1}^r b_ix^i\bigr)$ for which $a_1=b_1=1$.

Now, we can write $\psi((f)_p) = (\eta,\theta)$ for some jets $\eta,\theta\in J_r$, and since $f$ is differentiable at $p$ the coefficients of $x$ must be the same in $\eta$ and $\theta$.  It follows that the coefficient of $x$ in $\theta\circ \eta^{-1}$ is equal to~$1$.  By the previous paragraph, there exists a $(k)_p\in (G)_p'$ so that $\psi((k)_p)=(x,\theta\circ \eta^{-1})$, where $x$ denotes the identity element of~$J_r$.  Then $b=k^{-1}\circ f$ satisfies $\psi((b)_p) = (\eta,\eta)$.  In particular, $b$ is $C^r$ at~$p$, so $(b)_p\in (B)_p$.  Then $(f)_p=(k)_p(b)_p$, where $(k)_p\in (G)_p'$ and $(b)_p\in (B)_p$, so $(f)_p\in (G)_p'(B)_p$.
\end{proof}

\begin{proof}[Proof of Theorem~\ref{thm:PiecewiseCr}] Observe that $G = \PDiff_0^r(X)$ is a finite germ extension of $B = \Diff_0^r(X)$, with singular points at breakpoints. The action of $B$ on $X$ is transitive, so the orbits of $B$ and $G$ on $X$ are the same.  Since $B$ is simple (as $r\ne 2$), locally moving, and has no global fixed points in~$X$, it follows from Theorem~\ref{thm:Simplicity} that $G$ has simple commutator subgroup.  It remains to show that $G'=\ker(\varphi)$.

Fix a point $p\in X$.  By Lemma~\ref{lem:WhatIsNHere}, the group $A_p(X)$ is isomorphic to $\R$, with 
the corresponding quotient homomorphism $(G)_p\to \mathbb{R}$ given by
\[
(f)_p \mapsto \log\biggl(\frac{f^+(p)}{f^-(p)}\biggr).
\]
Since $B$ acts transitively, $R=\{p\}$ is a system of representatives for the orbits in~$X$.  Therefore, the germ abelianization homomorphism $\sigma$ defined in Theorem~\ref{thm:ActualVersion} is just $\sigma_p\colon G\to A_p(G)$.  But $\sigma_p$ is just the given homomorphism $\varphi\colon G\to\mathbb{R}$.  In particular, $\varphi$ satisfies the two conditions for $\sigma_p$ given in Theorem~\ref{thm:ActualVersion}, since $\varphi(f)=0$ if $f$ has no breakpoints, and $\varphi(f)$ is the image of $(f)_p$ in $A_p(G)=\mathbb{R}$ whenever $\sing(f)=\{p\}$.  Since $B = B'\leq G'$, it follows that $G'=G'B = \ker(\sigma)=\ker(\varphi)$ by Theorem~\ref{thm:ActualVersion}.
\end{proof}

\subsection{Further examples}\label{subsec:Examples}

In this section we give a few more examples of finite germ extensions, and describe how our theory applies to them. 

\begin{example}
The \newword{golden ratio Thompson group} $F_\varphi$ is the group of all piecewise-linear homeomorphisms of $[0,1]$ whose slopes are powers of the golden ratio $\varphi$, and whose breakpoints are in $\Z[\varphi]$. This was among the groups investigated by Melanie Stein~\cite{Stein}, and the group $F_\varphi$ itself was first considered by Sean Cleary~\cite{Cleary2}.  This group was further investigated by Jos\'{e} Burillo, Brita Nucinkis, and
Lawrence Reeves in~\cite{BNR1}, and related groups $T_\varphi$ and $V_\varphi$ were considered in~\cite{BNR2}.

The groups $F_\varphi$, $T_\varphi$, and $V_\varphi$ are finite germ extensions of Thompson's groups $F$, $T$, and $V$, respectively.  In particular, there is a tree of intervals $I_{\alpha}\subseteq [0,1]$ for $\alpha\in \{0,1\}^*$, defined recursively starting with $I_\emptyset = [0,1]$, where each interval $I_\alpha=[a,b]$ is subdivided into two subintervals
\[
I_{\alpha 0}=\bigl[a,\varphi^{-1}a+\varphi^{-2}b\bigr]\qquad\text{and}\qquad I_{\alpha 1} = \bigl[\varphi^{-1}a+\varphi^{-2}b,b\bigr].
\]
This defines a quotient map $q\colon\{0,1\}^\omega\to [0,1]$.  If we view Thompson's group $F$ as acting on $\{0,1\}^\omega$ in the usual way, then the quotient map $q$ determines an embedding of $F$ into $F_\varphi$, and it turns out that $F_\varphi$ is a finite germ extension of~$F$.  Similarly, the quotient map $q$ determines embeddings of $T$ into~$T_\varphi$ and $V$ into~$V_\varphi$, and these are finite germ extensions as well.  Using Theorems~\ref{thm:Abelianization} and \ref{thm:MainFinitenessTheorem}, it is possible to re-derive known results on the abelianizations of these groups, as well as prove that $F_\varphi$, $T_\varphi$, and $V_\varphi$ all have type~$\Finfty$.  Details for all of these arguments will appear elsewhere.
\end{example}

\begin{example}
Recall that Thompson's group $T$ acts by piecewise-linear homeomorphisms on the circle $\R/\Z$.  Let $S\subseteq \Q/\Z$ be any set of non-dyadic rational points which is invariant under the action of~$T$, e.g.~$S$ could be the orbit of $1/3$ under $T$. Let $G$ be the group of all piecewise-linear homeomorphisms of the circle that satisfy the following conditions:
\begin{enumerate}
    \item Each linear segment has the form $\theta \mapsto 2^n\theta + d$ for some $n\in\Z$ and some dyadic rational~$d$.\smallskip
    \item Each breakpoint occurs at either a dyadic rational or at an element of~$S$.
\end{enumerate}
Then $G$ is a finite germ extension of $T$ with singular set~$S$. Moreover, $(G)_s \cong \Z\oplus\Z$ for each $s\in S$ (corresponding to the left and right slopes at~$s$), with $(T)_s$ being the diagonal subgroup, so
\[
A_s(G) = (G)_s/(T)_s \cong \Z
\]
for each $s\in S$.  By Theorems~\ref{thm:Simplicity} and~\ref{thm:Abelianization}, every quotient of $G$ is abelian, the commutator subgroup $G'$ is simple, and $G/G'$ is a free abelian group whose rank is the number of orbits of $T$ in~$S$.
\end{example}

\begin{example}Theorem~\ref{thm:Abelianization} gives a new method for computing the abelianizations of R\"over--Nekrashevych groups associated to groups of bounded automata (see Section~\ref{sec:RoverNekrashevych}).  Another method for computing such abelianizations was previously given by Nekrashevych~\cite[Theorem~9.14]{Nek}, but requires knowing the abelianization of the underlying self-similar group.

For example, \newword{Grigorchuk's group} is the bounded automata group $\mathcal{G}\leq \mathrm{Aut}(\T_2)$ generated by elements $a,b,c,d\in\mathrm{Aut}(\T_2)$, which can be defined recursively by
\[
a(1)=b(0)=c(0)=d(0)=0,\qquad a(0)=b(1)=c(1)=d(1)=1,
\]
\[
a|_0=a|_1=d|_0=\mathrm{id},\qquad b|_0=c|_0=a,\qquad b|_1=c,\qquad c|_1=d,\qquad d|_1=b.
\]
The associated R\"over--Nekrashevych group is \newword{R\"over's group} $V\mathcal{G}$, first considered by Claas R\"over in~\cite{Rov}.  This group is a finite germ extension of $V$, with the singular set being the orbit of the point $p = \overline{1}=111\cdots$, and orbits under $V\mathcal{G}$ are the same as orbits under~$V$.  The group of germs $(V\mathcal{G})_p$ is generated by $(b)_p$, $(c)_p$, $(d)_p$ and $(\sigma)_p$, where $\sigma$ is any element of $V$ that maps the cone $C_{11}$ to the cone $C_{1}$ by a prefix replacement. The germs $(b)_p$, $(c)_p$, $(d)_p$ are the nontrivial elements of a Klein four-group $\mathbb{Z}_2\times \mathbb{Z}_2$, and conjugation by $(\sigma)_p$ permutes these cyclically, so  $(V\mathcal{G})_p$ is a semidirect product $(\mathbb{Z}_2\times \mathbb{Z}_2)\rtimes \mathbb{Z}$.  By Theorem~\ref{thm:Abelianization},  we get an exact sequence
\[
(V)_p/(V)'_p \longrightarrow (V\mathcal{G})_p\bigr/(V\mathcal{G})'_p \longrightarrow V\mathcal{G}/V\mathcal{G}' \longrightarrow 0.
\]
Here $(V)_p$ is the infinite cyclic group generated by $(\sigma)_p$, i.e.\ the factor of $\Z$ in $(\Z_2\times \Z_2)\rtimes \Z$.  It is easy to check that the inclusion $\Z\to (\Z_2\times \Z_2)\rtimes \Z$ induces an isomorphism on the abelianizations, so the first map above is an isomorphism, and therefore $V\mathcal{G}/V\mathcal{G}'$ is trivial.   This gives a new proof that R\"over's group is perfect, and combined with Theorem~\ref{thm:Simplicity}
 this yields a new proof that $V\mathcal{G}$ is simple.  
 
 Theorem~\ref{thm:RoverNek} also yields a new proof that~$V\mathcal{G}$ has type $\Finfty$. Indeed, the $\mathrm{CAT}(0)$ cubical complex $K$ described in Section~\ref{sec:GermComplex} is arguably simpler than the complex for $V\mathcal{G}$ described by the first and third authors in~\cite{BeMa}.  Note, however, that the action of $V\mathcal{G}$ on $K$ is not proper, so it remains an open question whether $V\mathcal{G}$ has the Haagerup property.
\end{example}

\section{Application to the Boone--Higman conjecture}\label{sec:BooneHigman} 

In this section we discuss the groups $\TA$ and $\VA$ defined in Section~\ref{subsec:IntroBooneHigman}.
We prove the essential properties of these groups in Section~\ref{subsec:MainTAVA}. In Section~\ref{subsec:TAVA2gen} we prove that $\TA$ and $\VA$ are 2-generated, while in Section~\ref{subsec:TAVAPresentations} we describe finite presentations for these groups.

\subsection{Main results on \texorpdfstring{$\boldsymbol{\TA}$}{TA} and \texorpdfstring{$\boldsymbol{\VA}$}{VA}}\label{subsec:MainTAVA}

\subsubsection*{Brin's group \texorpdfstring{$\A$}{A}}

Both of the  groups $\TA$ and $\VA$ are based on a certain group $\A\leq \Homeo(\R)$ introduced by Brin~\cite{Brin0}.  This can be described as the group of all homeomorphisms $f$ of $\R$ that satisfy the following conditions:
\begin{enumerate}
    \item $f$ is piecewise-linear with a locally finite set of breakpoints.\smallskip
    \item Each linear segment of $f$ has the form $t\mapsto 2^nt+b$ for some $n\in\Z$ and $b\in\Z\bigl[\tfrac12\bigr]$, and each breakpoint of $f$ lies in $\Z\bigl[\tfrac12\bigr]$.\smallskip
    \item There exists an $N>0$ so that $f(t+1)=f(t)+1$ whenever $|t|\geq N$.
\end{enumerate}
Though it will not be important for us, Brin showed that $\A$ is isomorphic to the group of ``orientation-preserving'' automorphisms of Thompson's group~$F$, with $\Aut(F)\cong \A\rtimes \Z_2$.

Brin proved that $\A$ fits into a short exact sequence
\[
F 
\,\hookrightarrow\, \A \,\twoheadrightarrow\, T\times T
\]
where the image of $F$ is the subgroup of $\A$ consisting of elements that are linear in neighborhoods of $\pm\infty$.  Since $F$ and $T$ are finitely presented, it follows that $\A$ is finitely presented. Explicit presentations of $\A$ have been given by Burillo and Cleary~\cite{BuClAuto} (see also Lemma~\ref{lem:PresentationA}).

Brin also showed that $\A$ fits into another short exact sequence:
\[
F' \hookrightarrow \A \twoheadrightarrow \Tbar\times\Tbar
\]
Here $\Tbar$ is the subgroup of $\A$ consisting of homeomorphisms $f\in\A$ that satisfy $f(t+1)=f(t)+1$ for all $t\in\R$.  The center $Z(\Tbar)$ is the infinite cyclic group generated by $t\mapsto t+1$, and the quotient $\Tbar/Z(\Tbar)$ is isomorphic to Thompson's group~$T$.  Each element of $\A$ agrees with some elements of $\Tbar$ near $-\infty$ and $+\infty$, and this defines the epimorphism $\A\twoheadrightarrow \Tbar\times \Tbar$.

In 2022, the authors proved that $\Q$ embeds into $\Tbar$, making $\Tbar$ the first explicit example of a finitely presented group that contains~$\Q$~\cite{BHM}.  Since $\Tbar$ is a subgroup of~$\A$, it follows that $\A$ contains $\Q$ as well.

\subsubsection*{Embedding $\A$ in $\TA$ and $\VA$}

To start, we describe a certain action of $\A$ on the interval $[0,1]$.  Consider the group $\A_{[0,1]}$ of all homeomorphisms $f$ of $[0,1]$ that satisfy the following conditions:
\begin{enumerate}
\item $f$ is piecewise-linear on $(0,1)$, though breakpoints may accumulate at $0$ and~$1$.\smallskip
\item Each linear segment of $f$ has the form $t\mapsto 2^nt+b$ for some $n\in\Z$ and $b\in\Z\bigl[\tfrac12\bigr]$, and each breakpoint of $f$ is in $\Z\bigl[\tfrac12\bigr]\cap (0,1)$.\smallskip
\item $L_0\circ f$ agrees with $f\circ L_0$ in a neighborhood of $0$, and $L_1\circ f$ agrees with $f\circ L_1$ in a neighborhood of~$1$, where $L_0(t)=2t$ and $L_1(t)=2t-1$.
\end{enumerate}
It is easy to prove that $\A_{[0,1]}$ is isomorphic to~$\A$.  Indeed, it is obtained by conjugating $\A$ by a certain piecewise-linear homeomorphism $\psi\colon \R\to (0,1)$ 
(see \cite[Proposition~3.1.1]{BelkBrown}).  

Identifying $0$ and $1$ in $\A_{[0,1]}$ gives a group $\A_{S^1}\leq \Homeo(S^1)$ which is isomorphic to~$\A$. Note that
\[
\A_{S^1} = \bigl\{f\in \TA \;\bigl|\; f(0)=0\text{ and }\sing(f)\subseteq\{0\}\bigr\}.
\]

Similarly, lifting $\A_{[0,1]}$ under the usual quotient map $\C_2\to [0,1]$ gives a group $\A_{\C_2}\leq \Homeo(\C_2)$ which is isomorphic to~$\A$, and
\[
\A_{\C_2} = \bigl\{f\in \VA \;\bigl|\; f\text{ is order-preserving and }\sing(f)\subseteq\{\overline{0},\overline{1}\}\bigr\}.
\]
In particular, both $\TA$ and $\VA$ contain isomorphic copies of~$\A$.

\begin{proposition}\label{prop:GeneratorsTAVA}
$\TA$ is generated by $T\cup \A_{S^1}$, and $\VA$ is generated by $V\cup \A_{\C_2}$.
\end{proposition}
\begin{proof}
If $f\in \TA$ and $\sing(f)=\emptyset$ then $f\in T$.  Suppose then that $f\in\TA$ and $|\sing(f)|\geq 1$.  Conjugating $f$ by a dyadic rotation from $T$, we may assume that $\sing(f)$ contains~$0$.  Let $g$ be an element of $\TA$ that agrees with $f$ in a neighborhood of~$0$ and satisfies $\sing(g)=\{0\}$.  Then $g\in \A_{S^1}$, and $\sing(g^{-1}f)=\sing(f)\setminus\{0\}$, so it follows by induction that $T\cup\A_{S^1}$ generates~$\TA$.

Similarly, if $f\in\VA$ and $\sing(f)=\emptyset$ then $f\in V$.  Suppose then that $f\in \VA$ and $|\sing(f)|\geq 1$.  Conjugating by an element of $V$, we may assume that $\sing(f)$ contains either $\overline{0}$ or $\overline{1}$.  Let $p\in\sing(f)\cap\{\overline{0},\overline{1}\}$ and let $g$ be an element of $\A_{\C_2}$ that agrees with $f$ in a neighborhood of $p$ and satisfies $\sing(g)=\{p\}$.  Then $\sing(g^{-1}f)=\sing(f)\setminus\{p\}$, so it follows by induction that $V\cup \A_{\C_2}$ generates~$\VA$.
\end{proof}

\begin{remark}
In general, if $G\leq \Homeo(X)$ is a finite germ extension of some base group~$B$ and the orbits of $G$ are the same as the orbits of~$B$, then $G$ is generated by $B$ together with subgroups $\SingFix_G(\{p\},\{p\})$, as $p$ ranges over representatives for the orbits of the singular points.
\end{remark}

\subsubsection*{Proofs of the main properties}

We are now ready to prove that $\TA$ and $\VA$ are finitely presented simple groups, that $\TA$ contains every countable, torsion free abelian group, and that $\VA$ contains every countable abelian group.

\begin{lemma}\label{lem:TbarPerfect}
The groups $\Tbar$ and $\A$ are perfect.
\end{lemma}
\begin{proof}
Note that any element of $\Tbar$ that fixes $0$ must fix every integer, and restriction to $[0,1]$ gives an isomorphism from $\Stab_{\Tbar}(0)$ to Thompson's group~$F$.  It is well-known that any element of $F$ that is the identity in neighborhoods of $0$ and $1$ lies in the commutator subgroup, so it follows that any element of $\Tbar$ which is the identity in a neighborhood of $0$ (and hence also $1$) lies in $\Tbar\,\!'$.

Let $f(t)=t+1/4$, and let $g$ be any element of $\Tbar$ which is the identity in a neighborhood of $0$ and agrees with $f$ in a neighborhood of $1/4$.  Then $k=f^{-2}gf$ is also the identity in a neighborhood of $0$, so both $g$ and $k$ lie in~$\Tbar\,\!'$.    We conclude that $f=g(fk^{-1}f^{-1})$ lies in $\Tbar\,\!'$ as well.  Since $f^4$ is the generator for $Z(\Tbar)$, we conclude that $Z(\Tbar)\leq \Tbar\,\!'$.  But $\Tbar/Z(\Tbar)\cong T$ is simple and hence perfect, so it follows that ${\Tbar}\,\!'=\Tbar$.

Finally, recall that $\A$ fits into a short exact sequence $F'\hookrightarrow \A\twoheadrightarrow \Tbar\times \Tbar$. It is well-known that $F'$ is perfect, and since $\Tbar\times \Tbar$ is perfect it follows that $\A$ is perfect.
\end{proof}

The proofs of the following two theorems use the groups of germs of $\TA$ and $\VA$.  It is easy to see that $(\A_{[0,1]})_0\cong (\A_{[0,1]})_1 \cong \Tbar$, with the natural epimorphism
\[
\A_{[0,1]}\twoheadrightarrow (\A_{[0,1]})_0\times (\A_{[0,1]})_1
\]
being a version of Brin's epimorphism $\A\twoheadrightarrow \Tbar\times\Tbar$.  It follows that $(\TA)_0\cong \Tbar\times \Tbar$ and $(\VA)_{\overline{0}}\cong (\VA)_{\overline{1}}\cong \Tbar$.  Note further that $(F)_0$ and $(F)_1$ are the subgroups of $(\A_{[0,1]})_0$ and $(\A_{[0,1]})_1$ corresponding to $Z(\Tbar)\cong\Z$.  It follows that $(T)_0$ is the center of $(\TA)_0$, and $(V)_{\overline{0}}$ and $(V)_{\overline{1}}$ are the centers of $(\VA)_{\overline{0}}$ and $(\VA)_{\overline{1}}$, respectively.

\begin{theorem}
$\TA$ and $\VA$ are simple.
\end{theorem}
\begin{proof}
It is well known that Thompson's group $T$ is simple, locally moving, and has no global fixed points in~$S^1$, so it follows from
Theorem~\ref{thm:Simplicity} that $\TA\hspace{0.08333em}'$ (the commutator subgroup of $\TA$) is simple. A similar argument shows that $\VA\hspace{0.08333em}'$ is simple.

By Proposition~\ref{prop:GeneratorsTAVA}, the group $\TA$ is generated $T$ and $\A_{S^1}\cong \A$.  Since $T$ is perfect and $\A$ is perfect by Lemma~\ref{lem:TbarPerfect}, it follows that $\TA$ is perfect and hence $\TA=\TA'$ is simple.  Similarly, $\VA$ is generated by the perfect subgroups $V$ and $\A_{\C_2}\cong \A$, so $\VA$ is perfect and hence simple.
\end{proof}

\begin{theorem}
$\TA$ and $\VA$ have type\/ $\Finfty$, and in particular are finitely presented.
\end{theorem}
\begin{proof}
We apply Theorem~\ref{thm:MainFinitenessTheorem}.  For $\TA$, recall that $\sing(\TA)$ is the set of dyadic points in~$S^1$.  Brown proved that $T$ has type~$\Finfty$~\cite{Bro1}.  Furthermore, if $M$ is a finite, nonempty subset of $\sing(\TA)$ then $\Stab_T(M)\cong F^{|M|}\rtimes \Z_{|M|}$.
Brown and Geoghegan proved that $F$ has type~$\Finfty$ \cite{BrGe}, so it follows that $\Stab_T(M)$ has type~$\Finfty$.  It is well-known that the induced action of $T$ on $\sing(\TA)^n$ has finitely many orbits for each $n\geq 1$.  Finally, if $p\in\sing(\TA)$ then $(\TA)_p\cong \Tbar\times\Tbar$, with $(T)_p$ corresponding to the normal subgroup $Z(\Tbar)\times Z(\Tbar)$.  Then $(\TA)_p/(T)_p\cong T\times T$, which has type~$\Finfty$.  By Theorem~\ref{thm:MainFinitenessTheorem}, we conclude that $\TA$ has type~$\Finfty$.

As for $\VA$, recall that $\sing(\VA)$ is the union of the $V$-orbits of $\overline{0}$ and $\overline{1}$.  Brown proved that $V$ has type~$\Finfty$~\cite{Bro1}, and $\Stab_V(M)$ has type $\Finfty$ for any finite, nonempty subset of $\sing(\VA)$ by Theorem~\ref{thm:Stabilizers}.  It is well-known that the induced action of $V$ on $\sing(\VA)^n$ has finitely many orbits for each $n\geq 1$.  Finally, if $p\in\sing(\VA)$ then $(\VA)_p\cong \Tbar$, with $(V)_p$ corresponding to the normal subgroup $Z(\Tbar)$, so $(\VA)_p/(V)_p\cong T$, which has type $\Finfty$.  By Theorem~\ref{thm:MainFinitenessTheorem}, we conclude that $\VA$ has type~$\Finfty$.
\end{proof}

\begin{theorem}\label{thm:AbelianSubgroupsTAVA}
$\TA$ contains\/ $\bigoplus_\omega \Q$, and $\VA$ contains $\bigoplus_{\omega}(\Q\oplus \Q/\Z)$.  Hence $\TA$ contains every countable, torsion-free abelian group, and $\VA$ contains every countable abelian group.
\end{theorem}
\begin{proof}
Observe that if $[a,b]$ is any arc in $S^1$ with dyadic endpoints, then the group
\[
\A_{[a,b]}=\bigl\{g\in \TA \;\bigr|\; g\text{ is supported on }[a,b]\text{ and }\sing(g)\subseteq \{a,b\}\bigr\}
\]
is isomorphic to~$\A$.  If $\{[a_n,b_n]\}_{n\in\N}$ is a pairwise disjoint collection of such intervals, the subgroup of $\TA$ generated by $\bigcup_{n\in\N} \A_{[a_n,b_n]}$ is isomorphic to $\bigoplus_\omega \A$.  Since $\A$ contains~$\Q$, it follows that $\TA$ contains $\bigoplus_\omega \Q$.

Similarly, for $\VA$, observe that for any cone $C_\alpha\subset \C_2$ the group
\[
\A_{C_\alpha} =\bigl\{g\in\VA \;\bigr|\; g\text{ is supported on }C_\alpha\text{ and }\sing(g)\subseteq \{\alpha\overline{0},\alpha\overline{1}\}\}
\]
is isomorphic to~$\A$, and the group $V_{C_\alpha}$ of elements of $V$ supported on $C_\alpha$ is isomorphic to~$V$.  If $\{C_{\alpha_n}\}_n\in\N$ is a pairwise disjoint collection of cones, then the subgroup of $\VA$ generated by $\bigcup_{n\in N}\bigl(\A_{C_{\alpha 0}}\cup V_{C_{\alpha1}}\bigr)$ is isomorphic to $\bigoplus_{\omega} (\A\oplus V)$.  But $\A$ contains~$\Q$, and Higman proved that $V$ contains $\Q/\Z$ \cite[Theorem~6.6]{Hig}, so it follows that $\VA$ contains $\bigoplus_\omega (\Q\oplus \Q/\Z)$. 
\end{proof}

\begin{remark}
There is another group related to $\VA$ which has similar properties but is arguably less complicated, namely the group $\VA_0$ of elements of $\VA$ whose singular points all lie in the orbit of~$\overline{0}$.  This group is also a finite germ extension of~$V$, and is simple and has type~$\Finfty$ just as~$\VA$. Furthermore, in the same way that $V$ is isomorphic to the group $V_2(S_2)$ (sometimes called ``$V$ with flips'') defined in~\cite{BDJ}, the group $\VA_0$ is isomorphic to the subgroup of $\Homeo(\C_2)$ generated by $V_2(S_2)$ and $\A_{\C_2}$.  It follows that $\VA$ embeds into $\VA_0$, so $\VA_0$ contains every countable abelian group.
\end{remark}

\subsection{2-generation}\label{subsec:TAVA2gen}

In this section we prove that the groups $\VA$ and $\TA$ are \mbox{2-generated}.  For $\VA$, this follows from a result of Bleak, Elliott, and Hyde~\cite[Theorem~1.12]{BlElHy}.

\begin{proposition}\label{prop:VA-2-gen}
The group $\VA$ is $2$-generated.
\end{proposition}
\begin{proof}
The paper \cite{BlElHy} considers the class of vigorous groups, i.e.\ groups of homeomorphisms of a Cantor space with the property that for every clopen set $A$ and all proper, nonempty clopen subsets $B$ and $C$ of $A$, there exists an element of the group which is supported on $A$ and maps $B$ into $C$. 
 The group $\VA$ is vigorous since Thompson's group~$V$ is vigorous and $\VA$ contains~$V$.  We know that $\VA$  is finitely generated and simple, so it follows from \cite[Theorem~1.12]{BlElHy} that $\VA$ can be generated by two elements of finite order, one of which has order two.
\end{proof}

Unfortunately, the theorem of Bleak, Elliott, and Hyde does not apply directly to~$\TA$, since $\TA$ is not vigorous.  However, the proof of their theorem can be adapted to the case of~$\TA$ (as well as many other groups of homeomorphisms of the circle).

\begin{theorem}\label{thm:TA-2-gen}
The group $\TA$ is\/ $2$-generated.
\end{theorem}
\begin{proof}
Since $\A$ is perfect (by Lemma~\ref{lem:TbarPerfect}) and finitely generated, there exist finitely many elements $f_1,\ldots,f_k\in \A$ such that the commutators $[f_i,f_j]$ ($i\ne j)$ generate~$\A$.  Choose an $n\in\N$ and integers $0\leq u_1 < u_2 < \cdots < u_k < n$ so that the differences $u_i-u_j$ ($i\ne j)$ are distinct modulo~$n$ (e.g.~$u_i=2^i$ and $n=2^{k+1}$).  Fix dyadic points $p_0,q_0,p_1,q_1,\ldots,p_{n-1},q_{n-1}$ on the circle in counterclockwise order, and let $t$ be an element of $T$ so that
\[
t(p_i) = q_i\qquad\text{and}\qquad t(q_i)=p_{i+1}
\]
for each $i$, where the subscripts are modulo~$n$. Let $r=t^2$, and note that each power~$r^i$ maps the interval $[p_0,p_1]$ to $[p_i,p_{i+1}]$.

Let $\psi\colon \A\to \A_{[p_0,p_1]}$ be an isomorphism, where $\A_{[p,q]}$ denotes the copy of $\A$ supported on an arc $[p,q]$ (see the proof of Theorem~\ref{thm:AbelianSubgroupsTAVA}), and let
\[
f = \bigl(r^{u_1}\psi(f_1)\,r^{-u_1}\bigr)\bigl(r^{u_2}\psi(f_2)\,r^{-u_2}\bigr)\cdots \bigl(r^{u_k}\psi(f_k)\,r^{-u_k}\bigr).
\]
Note then that $f$ is supported on $[p_{u_1},p_{u_1+1}]\cup \cdots \cup [p_{u_k},p_{u_k+1}]$, and agrees with $r^{u_i}\psi(f_i)\,r^{-u_i}$ on each $[p_{u_i},p_{u_{i+1}}]$. We claim that $t$ and $f$ generate $\TA$.

Observe that, since the differences $u_i-u_j$ ($i\ne j$) are distinct modulo~$n$, all of the intervals of support of $r^{-u_i}fr^{u_i}$ and $r^{-u_j}fr^{u_j}$ for $i\ne j$ are different except for $[p_0,p_1]$.  It follows that
\[
[r^{-u_i}fr^{u_i},r^{-u_j}fr^{u_j}] = \psi\bigl([f_i,f_j]\bigr)
\]
for each $i\ne j$.  These commutators generate $\A_{[p_0,p_1]}$, and therefore $\A_{[p_0,p_1]}$ is contained in the subgroup generated by $t$ and~$f$. Conjugating by powers of $t$, we deduce that each $\A_{[p_i,p_{i+1}]}$ as well as each $\A_{[q_i,q_{i+1}]}$ lies in the subgroup generated by $t$ and~$f$.

Now, $T$ is generated by the elements of $T$ supported on each~$[p_i,p_{i+1}]$ and each $[q_i,q_{i+1}]$ (since the interiors of these intervals cover the circle), so $T\leq \langle t,f\rangle$.  But the group
\[
\A_{p_1} = \bigl\{g\in\TA \;\bigr|\; g(p_1)=p_1\text{ and }\sing(g)\subseteq\{p_1\}\bigr\}
\]
is generated by the elements of $T$ that fix $p_1$ together with the elements of $\A_{[p_0,p_1]}\cap \A_{p_1}$ and $\A_{[p_1,p_2]}\cap \A_{p_1}$, and therefore $\A_{p_1}\leq \langle t,f\rangle$ as well.  Conjugating $\A_{p_1}$ by an element of $T$ that maps $p_1$ to $0$, we conclude that $\A_{S^1}\leq \langle t,f\rangle$, so $t$ and $f$ generate~$\TA$ by Proposition~\ref{prop:GeneratorsTAVA}.
\end{proof}

\subsection{Presentations}\label{subsec:TAVAPresentations}

In this section we use the geometry of the germ complex to describe presentations for $\TA$ and~$\VA$.  In this section, if $G\leq\Homeo(X)$ is a finite germ extension and $M\subseteq X$, we let
\[
\SingStab_G(M) = \{g\in G\mid g(M)=M\text{ and }\sing(g)\subseteq M\}.
\]
We begin by counting generators and relations in the group~$\A$.  An explicit presentation for $\A$ with 8 generators and 35 relations was given by Burillo and Cleary in~\cite[Proposition~3.1]{BuClAuto}.  The following lemma improves slightly on their result by using the presentation for $T$ given by Lochak and Schneps~\cite{LoSc}.

\begin{lemma}\label{lem:PresentationA}
The group $\A$ has a presentation with\/ $6$ generators and\/ $24$ relations.
\end{lemma}
\begin{proof}
Recall that $\A$ fits into a short exact sequence $F\hookrightarrow\A\twoheadrightarrow T\times T$.  Thompson's group $F$ has a well-known presentation with $2$ generators and $2$ relations. The smallest known presentation for Thompson’s group $T$ has $2$ generators and $5$ relators, and was derived by Lochak and Schneps\footnote{Note that there is a typo in the presentation stated in~\cite{LoSc}. See~\cite[Proposition 1.3]{FunKap}
%[7, Proposition 1.3] 
for a correct version.} in~\cite{LoSc}.
%Furthermore, Lochack and Schneps found a presentation for $T$ with $2$ generators and $5$~relations \blue{reference, including typo}.

It is proven in~\cite[Proposition~2.55]{HoltHandbook} that if $G=\langle X\mid R\rangle$ and $H=\langle Y\mid S\rangle$ are finitely presented groups, then any extension of $G$ by $H$ has a presentation with $|X|+|Y|$ generators and $|R|+|S|+|X|\,|Y|$ relations.  Since $T\times T$ is an extension of $T$ by $T$, it follows that $T\times T$ has a presentation with $2+2=4$ generators and $5+5+(2)(2)=14$~relations,
and therefore $\A$ has a presentation with $4+2=6$~generators and $14+2+(4)(2)=24$~relations. 
\end{proof}

\begin{theorem}\label{thm:PresentationTA}
The group $\TA$ is the amalgamated sum of Thompson's group $T$, the subgroup $\A_{S^1}\cong \A$, and the subgroup\/ $\SingStab_{\TA}(\{0,1/2\}) \cong \A\wr\Z_2$, and it has a presentation with\/ $2$ generators and\/ $90$ relations.
\end{theorem}
\begin{proof}
Let $K$ be the germ complex associated to $\TA$.  By Proposition~\ref{prop:nConnected}, the sublevel complex $K_{\leq 2}$ is simply connected.  Let $K'$ be the simplicial complex obtained from $K_{\leq 2}$ by subdividing each square diagonally from the vertex with no hidden points to the vertex with two hidden points, and note that $\TA$ acts simplicially and rigidly on~$K'$.  Let $\Delta$ be the $2$-simplex in $K'$ whose vertices are the base vertex~$v_0$, the adjacent vertex $v_1$ for which $0$ is hidden, and the adjacent vertex $v_2$ for which $0$ and $1/2$ are hidden. Since $T$ acts $2$-transitively on dyadic points, every non-singular simplex of $K'$ lies in the $T$-orbit of some face of~$\Delta$, and hence every simplex of $K'$ lies in the $\TA$-orbit of some face of $\Delta$, i.e.\ $\Delta$ is a fundamental domain for the action of $\TA$. By a theorem of Brown~\cite[Theorem~3]{Bro2} it follows that $\TA$ is the sum of the stabilizers of $v_0$, $v_1$, and $v_2$, amalgamated over their intersections.  But $\Stab_{\TA}(v_0)=T$, $\Stab_{\TA}(v_1)=\A_{S^1}$, and $\Stab_{\TA}(v_2)=\SingStab_{\TA}(\{0,1/2\})$, so $\TA$ is the amalgamated sum of these three subgroups.

To compute a presentation for $\TA$, observe that if $G$ is an amalgamated sum of subgroups $H_1,\ldots,H_n$ with presentations  $H_i=\langle X_i\mid R_i\rangle$, and each $H_i\cap H_j$ has generating set $Y_{ij}$, then $G$ has a presentation with $\sum_i |X_i|$ generators and $\sum_i |R_i| + \sum_{i<j} |Y_{ij}|$ relations. As mentioned above, Lochak and Schneps found a presentation for $T$ with $2$ generators
and $5$ relations, and by Lemma~\ref{lem:PresentationA} the group $\A$ has a presentation with $6$ generators $x_1,\ldots,x_6$ and $24$ relations. Finally,
\[
\SingStab_{\TA}(\{0,1/2\})\cong \A\wr \Z_2 \cong \langle \A,t\mid t^2,[x_i,tx_jt]\text{ for }i\leq j\rangle
\]
has $7$ generators and $46$ relations.  The three intersections are:
\begin{enumerate}
    \item $T\cap \A_{S_1}\cong F$, which is generated by $2$ elements.\smallskip
    \item $T\cap \SingStab_{\TA}(\{0,1/2\}) = \Stab_T(\{0,1/2\})\cong F\wr\Z_2$.  Since $F$ is generated by two elements, this is generated by $3$ elements.\smallskip
    \item $\A_{S^1}\cap \SingStab_{\TA}(\{0,1/2\}) = \Stab_{\A_{S^1}}(1/2)$.  For this group, the kernel of the epimorphism
    \[
    \Stab_{\A_{S^1}}(1/2) \twoheadrightarrow (\A_{S^1})_0 \cong \Tbar\times \Tbar \twoheadrightarrow T\times T
    \]
    is the subgroup of $T$ that fixes $\{0,1/2\}$ pointwise, which is isomorphic to $F\times F$.  Since $T\times T$ and $F\times F$ are each generated by $4$ elements, it follows that this intersection is generated by $8$ elements.
\end{enumerate}
We conclude that $\TA$ has a presentation with $2+6+7=15$ generators and $(5+24+46)+(2+3+8)=88$ relations.

By Theorem~\ref{thm:TA-2-gen} the group $\TA$ is $2$-generated.  We can use Tietze transformations to replace the $15$-element generating set by a $2$-element generating set at the cost of $2$ more relations, so $\TA$ has a presentation with $2$ generators and $90$ relations.
\end{proof}

\begin{remark}
The argument at the beginning of the proof of Theorem~\ref{thm:PresentationTA} is very general.  Indeed, if $G\leq\Homeo(X)$ is a finite germ extension and the base group $B$ acts $2$-transitively on $\sing(G)$, then $G$ is the amalgamated sum of $B$, $\SingStab_G(\{p\})$, and $\SingStab_G(\{p,q\})$, where $p$ and $q$ are a pair of distinct points in~$\sing(G)$.  This applies, for example, to R\"over's group, though explicit presentations for R\"over--Nekrashevych groups have already been derived by Nekrashevych~\cite{Nek2}.
\end{remark}

\begin{remark}
Except for the last step where the $15$-element generating set was replaced by the $2$-element generating set, all of the steps of the proofs of Lemma~\ref{lem:PresentationA} and Theorem~\ref{thm:PresentationTA} were quite explicit and could be carried out by hand to derive an explicit presentation of $\TA$ with $15$ generators and $88$ relations.
\end{remark}

\begin{theorem}
The group $\VA$ is the amalgamated sum of the family of subgroups\/ $\SingStab_{\VA}(M)$ for $M\in\bigl\{\emptyset,
\{\overline{0}\},
\{\overline{0},1\overline{0}\},
\{\overline{1}\},
\{0\overline{1},\overline{1}\},
\{\overline{0},\overline{1}\}
\bigr\}$, and has a presentation with\/ $2$ generators and\/
$239$ relations.
\end{theorem}
\begin{proof}
As in the proof of Theorem~\ref{thm:PresentationTA}, let $K_{\leq 2}$ be the sublevel-2 subcomplex of the germ complex for $\VA$, and let $K'$ be the simplicial subdivision of $K_{\leq 2}$ obtained by cutting squares along diagonals.  Let $L$ be the induced subcomplex of $K'$ whose vertex set consists of the six non-singular partial portraits whose sets of hidden points are $\emptyset$, $\{\overline{0}\}$,
$\{\overline{0},1\overline{0}\}$,
$\{\overline{1}\}$, $\{0\overline{1},\overline{1}\}$, and
$\{\overline{0},\overline{1}\}$ (so $L$ has $6$ vertices, $9$ edges, and $4$ triangles).  Again, $L$ is a fundamental domain for the action of~$\VA$, so by Brown's theorem \cite[Theorem~3]{Bro2}
the group $\VA$ is the amalgamated sum of the stabilizers of the vertices of~$L$, which are precisely the subgroups $\SingStab_{\VA}(M)$.

To derive a presentation for $\VA$, observe that $\SingStab_{\VA}(\emptyset)=V$, and each of the remaining subgroups $\SingStab_{\VA}(M)$ has $\Fix_V(M)$ as a normal subgroup, which by Theorem~\ref{thm:Stabilizers} is an $|M|$-fold ascending HNN extension of~$V$.  Any ascending HNN extension of a group $G=\langle X\mid R\rangle$ has a presentation with $|X|+1$ generators and $|X|+|R|$ relations, and  Bleak and Quick have shown that $V$ has a presentation with $2$ generators and $7$ relations~\cite[Theorem~1.3]{BlQu}, so it follows that each $\Fix_V(M)$ has a presentation with $|M|+2$ generators and $\binom{|M|+2}{2}+6$ relations.  The following tables summarize presentations for the given subgroups, with the first table showing some preliminary calculations:
\[
\renewcommand{\arraystretch}{1.25}
\begin{array}{cccl}
\text{Group} & \text{Gens.} & \text{Rels.} & \text{Notes} \\
\hline
T & 2 & 5 & \text{Lochak and Schneps \cite{LoSc}} \\
T\wr\Z_2 & 3 & 9 & \langle T,t\mid t^2,[a,tat],[a,tbt],[b,tbt]\rangle \\
T\times T & 4 & 14 & \text{extension of $T$ by $T$}
\end{array}
\]
%\mage{ERASE THIS: $\Fix_V(\{a\})$ has $1+2=3$ generators and $\binom{1+2}{2}+6=3+6=9$ relations. $\Fix_V(\{a,b\})$ has $2+2=4$ generators and$\binom{2+2}{2}+6=12$ relations.}
\[
\renewcommand{\arraystretch}{1.25}
\begin{array}{cccl}
\text{Group} & \text{Gens.} & \text{Rels.} & \text{Notes} \\
\hline
V = \SingStab_{\VA}(\emptyset) & 2 & 7 & \text{Bleak and Quick \cite{BlQu}} \\
\SingStab_{\VA}(\{\overline{0}\}) & 5 & 20 & \text{extension of $T$ by $\Fix_V(\{\overline{0}\})$} \\
\SingStab_{\VA}(\{\overline{1}\}) & 5 & 20 & \text{extension of $T$ by $\Fix_V(\{\overline{1}\})$} \\
\SingStab_{\VA}(\{\overline{0},1\overline{0}\}) & 7 & 33 & \text{extension of $T\wr\Z_2$ by $\Fix_V(\{\overline{0},1\overline{0}\})$}
\\
\SingStab_{\VA}(\{0\overline{1},\overline{1}\}) & 7 & 33 & \text{extension of $T\wr\Z_2$ by $\Fix_V(\{0\overline{1},\overline{1}\})$}
\\
\SingStab_{\VA}(\{\overline{0},\overline{1}\}) & 8 & 42 & \text{extension of $T\times T$ by $\Fix_V(\{\overline{0},\overline{1}\})$} \\
\end{array}
\]
These six subgroups have fifteen possible intersections, whose types and sizes of generating sets are summarized in the following table:
\[
\renewcommand{\arraystretch}{1.25}
\begin{array}{cccl}
\text{Type} &  \text{No.} & \text{Gens.} & \text{Notes} \\
\hline
\Stab_V(\{p\}) & 2 & 3 & \text{same as } \Fix_V(\{p\}) \\
\Stab_V(\{p,q\}) & 2 & 4 & \text{same as } \Fix_V(\{p,q\})\\
\Stab_V(\{p,p'\}) & 2 & 5 & \text{extension of } \Z_2 \text{ by } \Fix_V(\{p,p'\}) \\
\Stab_V(\{p,p',q\}) & 2 & 6 & \text{extension of } \Z_2 \text{ by } \Fix_V(\{p,p',q\}) \\
\Stab_V(\{p,p',q,q'\}) & 1 & 8 & \text{extension of } \Z_2\times\Z_2 \text{ by } \Fix_V(\{p,p',q,q'\}) \\
\SingFix_{\VA}(\{p\},\{p,p'\}) & 2 & 6  & \text{extension of } T \text{ by } \Fix_V(\{p,p'\}) \\
\SingFix_{\VA}(\{p\},\{p,q\}) & 2 & 6 & \text{extension of } T \text{ by } \Fix_V(\{p,q\}) \\
\SingFix_{\VA}(\{p\},\{p,p',q\}) & 2 & 7 & \text{extension of } T \text{ by } \Fix_V(\{p,p',q\})
\end{array}
\]
Here $p$ and $p'$ represent distinct points in the same $V$-orbit, as do $q$ and $q'$, but $p$ and $q$ are in different orbits.
% stab 0:  3 generators
% stab 1:  3 generators
% stab {0,10}:  4 generators
% stab {01,1}:  4 generators
% fix {0,1}:  4 generators
% singfix({0},{0,10}):  6 generators
% fix {0,1}:  4 generators
% stab {0,01,1}:  5 generators
% singfix({0},{0,1}):  6 generators
% stab {0,10,1}:  5 generators
% stab {0,10,01,1}:  6 generators
% singfix({0},{0,10,1}):  7 generators
% singfix({1},{01,1}):  6 generators
% singfix({1},{0,1}):  6 generators
% singfix({1},{0,01,1}):  7 generators
We conclude that $\VA$ has a presentation with $2 + 5 + 7 + 5  + 7 + 8=34$ generators and
\[
(7+20+20+33+33+42)+2(3)+2(4)+2(5)+2(6)+1(8)+2(6)+2(6)+2(7) = 237
\]
relations.

By Proposition~\ref{prop:VA-2-gen} the group $\VA$ is $2$-generated, so it follows that $\VA$ has a presentation with $2$ generators and $239$ relations.
\end{proof}

\appendix

\section{Stabilizers in Higman--Thompson groups}\label{sec:AppendixA}

For $d\geq 2$ and $r\geq 1$, let $\C_{d,r}$ denote the Cantor space $X_r\times X_d^\omega$, and let $V_{d,r}$ be the corresponding Higman--Thompson group.  Recall that a point in $\C_{d,r}$ is \newword{rational} if it is eventually periodic, i.e.\ it has the form $\alpha\overline{\beta}$ for some $\alpha\in X_r\times X_d^*$ and $\beta\in X_d^+$.  In this appendix we prove the following theorem.  Additional results about stabilizers in Higman--Thompson groups can be found in~\cite{BHM2}.

\begin{theorem}\label{thm:Stabilizers}
Let $S\subset \C_{d,r}$ be a finite set of rational points. 
 Then\/ $\Fix_{V_{d,r}}(S)$ is an iterated ascending HNN extension of $V_{d,n}$ for some $n\geq 1$, and in particular\/ $\Fix_{V_{d,r}}(S)$ and\/ $\Stab_{V_{d,r}}(S)$ have type\/~$\Finfty$.
\end{theorem}

Here a group $G$ is an \newword{ascending HNN~extension} of a subgroup~$H$ if $G$ has presentation
\[
G = \langle H,t \mid t^{-1}ht=\varphi(h)\text{ for all }h\in H\rangle
\]
for some $t\in G$ and some monomorphism $\varphi\colon H\to H$.  This is equivalent to the existence of an element $t \in G$ that satisfies the following conditions:
\begin{enumerate}
    \item $t^i\notin H$ for all $i\geq 1$,\smallskip
    \item $t^{-1}Ht\leq H$, and\smallskip
    \item $\bigcup_{i,j\in\N} t^i H t^{-j}= G$.
\end{enumerate}
Theorem~\ref{thm:Stabilizers} asserts that $\Fix_{V_{d,r}}(S)$ is an iterated ascending HNN extension of some $V_{d,n}$, i.e.\ there exists a chain of subgroups
\[
V_{d,n}\cong H_0 \leq H_1 \leq \cdots \leq H_k = \Fix_{V_{d,r}}(S)
\]
such that each $H_i$ is an ascending HNN extension of~$H_{i-1}$. Note that $\Fix_{V_{d,r}}(S)$ has finite index in $\Stab_{V_{d,r}}(S)$, so the finiteness properties of these two groups are the same (see \cite[Proposition~7.2.3]{Geog}). 

\begin{proof}[Proof of Theorem~\ref{thm:Stabilizers}]
Let $G=\Fix_{V_{d,r}}(S)$. 
 Brown proved that the groups $V_{d,n}$ all have type~$\Finfty$ \mbox{\cite[Theorem~7.3.1]{Bro1}}.  It is well-known that an ascending HNN extension of a group of type~$\Finfty$ has type~$\Finfty$, e.g.\ by applying \cite[Proposition~1.1]{Bro1} to the action on the Bass--Serre tree.  Therefore, it suffices to prove that $G$ is an ascending HNN extension of $V_{d,n}$ for some~$n$. We proceed by induction on~$|S|$.  The base case is $|S|=0$, for which the statement is trivially true.

Now suppose $|S|>0$, and fix a point $s\in S$.  Since $s$ is rational, we can write $s$ as $\alpha\overline{\beta}$, where $\alpha\in X_d\times X_r^*$, $\beta\in X_r^+$, and $\beta$ is not a power of any shorter word.  Let $t$ be an element of $G$ that agrees with the prefix replacement $\alpha\psi\mapsto \alpha\beta\psi$ in a neighborhood of~$s$, and note that the group of germs $(G)_s$ is the infinite cyclic group generated by~$(t)_s$. 
Since $s$ is an attracting fixed point for $t$, we can find a clopen neighborhood $A$ of $s$ such that $t(A)\subseteq A$ and $\bigcap_{n\in \N} t^n(A)=\{s\}$.

Let $E=\C_{d,r}\setminus A$, and let $H$ be the subgroup of $G$ consisting of elements that are supported on $E$, i.e.\ $H=\Fix_G(A)$. 
Note that the subgroup of $V_{d,r}$ of elements supported on $E$ is isomorphic to (indeed, conjugate to) $V_{d,m}$ for some $m\geq 1$, and $H$ is the subgroup of this copy of $V_{d,m}$ consisting of elements that fix $S\setminus\{s\}$ pointwise.  By our induction hypothesis, it follows that $H$ is an iterated ascending HNN extension of $V_{d,n}$ for some $n\geq 1$.  Therefore, it suffices to prove that $G$ is an ascending HNN extension of~$H$.

We verify the three conditions for ascending HNN extensions listed above.  For~(1), since elements of $H$ are the identity on $A$, no positive power of $t$ can be contained in~$H$.  For~(2), observe that any element of $t^{-1}Ht$ is supported on $t^{-1}(E)$.  Since $t(A)\subseteq A$ we have that $t^{-1}(E)\subseteq E$, and hence $t^{-1}Ht\leq H$.  Finally, for~(3), if $g\in G$ then since $(G)_s=\langle (t)_s\rangle$ we know that $(g)_s=(t)_s^i$ for some integer $i\in\Z$.  Let $U$ be a neighborhood of $s$ on which $g$ agrees with $t^i$, and let $j\geq |i|$ so that $t^j(A)\subseteq U$.  Since $t^{-i}g$ is the identity on $U$, it follows that $t^{-j}(t^{-i}g)t^j= t^{-i-j}gt^j$ is the identity on~$A$.  Then $t^{-i-j}gt^j$ lies in $H$, and therefore $g\in t^{i+j}Ht^{-j}$.  We conclude that $G$ is an ascending HNN extension of $H$, and therefore $G$ is an iterated ascending HNN extension of some~$V_{d,n}$.
\end{proof}

\begin{remark}
As long as $S$ is nonempty, the number $n\geq 1$ in Theorem~\ref{thm:Stabilizers} can actually be chosen to be any desired positive integer. Specifically, in the case where $|S|=1$ we can choose $A$ so that $\Fix_{V_{d,r}}(A)\cong V_{d,n}$ for any desired~$n$, and the $|S|>1$ case follows.  See \cite{BHM2} for details.
\end{remark}

\bigskip
\newcommand{\arXiv}[1]{\href{https://arxiv.org/abs/#1}{arXiv}}
\newcommand{\doi}[1]{\href{https://doi.org/#1}{Crossref\,}}
\bibliographystyle{plain}

\end{document}